\documentclass[11pt,twoside]{article}
\RequirePackage[OT1]{fontenc}
\RequirePackage{amsthm,amsmath}

\usepackage{amssymb}
\usepackage{amsmath}
\usepackage{amsfonts}
\usepackage{subfigure}
\usepackage{multirow}
\usepackage{xr}
\usepackage{float}
\usepackage{xspace}
\usepackage{algorithm}
\usepackage{algpseudocode}
\usepackage{algorithmicx}
\usepackage{color}
\usepackage{cite}
\usepackage{array}
\usepackage{color}
\usepackage{amsmath,amssymb}

\usepackage{natbib}
\usepackage{upgreek}
\usepackage{amsthm,amssymb,amsmath}
\usepackage{bm}
\usepackage{mathrsfs}
\usepackage{ctable}
\date{}
\newtheorem{theorem}{Theorem}
\newtheorem{lemma}{Lemma}
\newcommand{\Var}{\mathrm{Var}}
\newcommand{\E}{\mathrm{E}}
\title{\textbf{Bayesian estimation in differential equation models}}
\author{Prithwish Bhaumik and Subhashis Ghosal\vspace{.1in}\\\it{North Carolina State University}}
\begin{document}
\maketitle
\begin{abstract}
Ordinary differential equations (ODEs) are used to model dynamic systems appearing in engineering, physics, biomedical sciences and many other fields. These equations contain unknown parameters, say $\bm\theta$ of physical significance which have to be estimated from the noisy data. Often there is no closed form analytic solution of the equations and hence we cannot use the usual non-linear least squares technique to estimate the unknown parameters. There is a two step approach to solve this problem, where the first step involves fitting the data nonparametrically. In the second step the parameter is estimated by minimizing the distance between the nonparametrically estimated derivative and the derivative suggested by the system of ODEs. The statistical aspects of this approach have been studied under the frequentist framework. We consider this two step estimation under the Bayesian framework. The response variable is allowed to be multidimensional and the true mean function of it is not assumed to be in the model. We induce a prior on the regression function using a random series based on the B-spline basis functions. We establish the Bernstein-von Mises theorem for the posterior distribution of the parameter of interest. Interestingly, even though the posterior distribution of the regression function based on splines converges at a rate slower than $n^{-1/2}$, the parameter vector $\bm\theta$ is nevertheless estimated at $n^{-1/2}$ rate.
\end{abstract}

\section{Introduction}

Suppose that we have a regression model $\bm Y=\bm f_{\bm\theta}(t)+\bm\varepsilon,\,\bm\theta\in\Theta\subseteq\mathbb{R}^p$. The explicit form of $\bm f_{\bm\theta}(\cdot)$ may not be known, but the function is assumed to satisfy the system of ordinary differential equations (ODEs) given by
\begin{eqnarray}
\frac{d\bm f_{\bm\theta}(t)}{dt}=\bm F(t,\bm f_{\bm\theta}(t),\bm\theta),\,t\in [0,1];\label{intro}
\end{eqnarray}
here $\bm F$ is a known appropriately smooth vector valued function and $\bm\theta$ is a parameter vector controlling the regression function. Equations of this type are encountered in various branches of science such as in genetics \citep{chen1999modeling}, viral dynamics of infectious diseases [\citet{anderson1992infectious}, \citet{nowak2000virus}]. There are numerous applications in the fields of pharmacokinetics and pharmacodynamics (PKPD) as well. For the two-compartment models a conservation law is described by the differential equations
\begin{eqnarray}
V_c\frac{dC_p(t)}{dt}&=&-Cl_d(C_p(t)-C_t(t))-ClC_p(t),\nonumber\\
V_t\frac{dC_t(t)}{dt}&=&Cl_d(C_p(t)-C_t(t)),\nonumber
\end{eqnarray}
where $C_p(t)$ and $C_t(t)$ denote the concentrations of drug in the central and peripheral compartments respectively, $V_c$ and $V_t$ stand for the volumes of the two compartments, and $Cl$ and $Cl_d$ are the clearance rates. The above system can be solved analytically, but there are a lot of instances where no closed form solution exist.
Such an example can be found in the feedback system \citep[page 235]{gabrielsson2000pharmacokinetic} modeled by the ODEs
\begin{eqnarray}
\frac{dR(t)}{dt}&=&k_{in}-k_{out}R(t)(1+M(t)),\nonumber\\
\frac{dM(t)}{dt}&=&k_{tol}(R(t)-M(t)),\nonumber
\end{eqnarray}
where $R(t)$ and $M(t)$ stand for loss of response and modulator at time $t$ respectively. Here $k_{in}, k_{out}$ and $k_{tol}$ are unknown parameters which have to be estimated from the noisy observations given by
\begin{eqnarray}
Y_R(t)&=&R(t)+\varepsilon_R(t),\nonumber\\
Y_M(t)&=&M(t)+\varepsilon_M(t),\nonumber
\end{eqnarray}
$\varepsilon_R(t), \varepsilon_M(t)$ being the respective noises at time point $t$. Another popular example is the Lotka-Volterra equations, also known as predator-prey equations. The prey and predator populations change over time according to the equations
\begin{eqnarray}
\frac{dp_1(t)}{dt}&=&\alpha p_1(t)+\beta p_1(t)p_2(t),\nonumber\\
\frac{dp_2(t)}{dt}&=&\gamma p_2(t)+\delta p_1(t)p_2(t),\nonumber
\end{eqnarray}
where $p_1(t)$ and $p_2(t)$ denote the prey and predator populations at time $t$ respectively.\\
\indent If the ODEs can be solved analytically, then the usual non linear least squares (NLS) [\citet{levenberg1944method}, \citet{marquardt1963algorithm}] can be used to estimate the unknown parameters. In most of the practical situations, such closed form solutions are not available as evidenced in the previous two examples. NLS was modified for this purpose by \citet{bard1974nonlinear}, \citet{domselaar1975nonlinear} and \citet{benson1979parameter}. They took initial values of the parameters and integrated both sides of \eqref{intro} to obtain $\bm f_{\bm\theta}$. Then $\bm\theta$ was estimated by minimizing the deviation of the observations from the corresponding $\bm f_{\bm\theta}$. These approaches suffer from the drawback of depending heavily on the initial choices of the parameters. \citet{hairer1987n} and \citet{mattheij1996ordinary} used the 4-stage Runge-Kutta algorithm as an alternative approach. The statistical properties of the corresponding estimator have been studied by \citet{xue2010sieve}. The strong consistency, $\sqrt{n}$-consistency and asymptotic normality of the estimator were established in their work.\\
\indent \citet{ramsay2007parameter} proposed the generalized profiling procedure where the solution is approximated by a linear combination of basis functions. The coefficients of the basis functions are estimated by solving a penalized optimization problem using an initial choice of the parameters of interest. A data-dependent fitting criterion is constructed which contains the estimated coefficients. Then ${\bm\theta}$ is estimated by the maximizer of this criterion. \citet{qi2010asymptotic} explored the statistical properties of this estimator including $\sqrt{n}$-consistency and asymptotic normality. Despite having desirable statistical properties, these approaches are computationally cumbersome especially for high-dimensional systems of ODEs as well as when $\bm\theta$ is high-dimensional.\\
\indent \citet{varah1982spline} introduced an approach of two step procedure. In the first step each of the state variables is approximated by a cubic spline using least squares technique. In the second step, the corresponding derivatives are estimated by differentiating the nonparametrically fitted curve and the parameter estimate is obtained by minimizing the sum of squares of difference between the derivatives of the fitted spline and the derivatives suggested by the ODEs at the design points. This method does not depend on the initial or boundary conditions of the state variables and is computationally quite efficient. An example given in \citet{voit2004decoupling} showed the computational superiority of the two stage approach over the usual least squares technique. \citet{brunel2008parameter} replaced the sum of squares of the second step by a weighted integral of the squared deviation and proved $\sqrt{n}$-consistency as well as asymptotic normality of the estimator so obtained. The order of the B-spline basis was determined by the smoothness of $\bm F(\cdot,\cdot,\cdot)$ with respect to its first two arguments. \citet{gugushvili2012n} used the same approach but used kernel smoothing instead of spline. They also established $\sqrt{n}$-consistency of the estimator. Another modification has been made in the work of \citet{wu2012numerical}. They have used penalized smoothing spline in the first step and numerical derivatives instead of actual derivatives of the nonparametrically estimated functions.\\
\indent In ODE models Bayesian estimation was considered in the works of \citet{gelman1996physiological}, \citet{girolami2008bayesian}. First they solved the ODEs numerically to approximate the expected response and hence constructed the likelihood. A prior was assigned on $\bm\theta$ and MCMC technique was used to generate samples from the posterior. Computation cost might be an issue in this case as well. The theoretical aspects of Bayesian estimation methods have not been explored in the literature.\\
\indent In this paper we consider a Bayesian analog of the approach of \citet{brunel2008parameter} fitting a nonparametric regression model using B-spline basis. A prior has been induced on $\bm\theta$ using the prior assigned on the coefficients of the basis functions. In this paper we study the asymptotic properties of the posterior distribution of $\bm\theta$ and establish a Bernstein-von Mises theorem. We let the ODE model to be misspecified, that is, the true regression function may not be a solution of the ODE. The response variable is also allowed to be multidimensional with possibly correlated errors. Normal distribution is used as the working model for error distribution, but the true distribution of errors may be different. Interestingly, the original model is parametric but it is embedded in a nonparametric model, which is further approximated by high dimensional parametric models. Note that the slow rate of nonparametric estimation does not influence the convergence rate of the parameter in the original parametric model.\\
\indent The paper is organized as follows. Section 2 contains the description of the notations and the model as well as the priors used for the analysis. The main results are given in Section 3. We extend the results to a more generalized set up in Section 4. In Section 5 we have carried out a simulation study under different settings. Proofs are given in Section 6.
\section{Notations, model assumption and prior specification}
We describe a set of notations to be used in this paper. For a matrix $\bm A$, the symbols $\bm A_{i,}$ and $\bm A_{,j}$ stand for the $i^{th}$ row and $j^{th}$ column of $\bm A$ respectively. We use the notation $\bm A_{r:s,}$ with $r<s$ to denote the sub-matrix of $\bm$ A consisting of $r^{th}$ to $s^{th}$ rows of $\bm A$. Similarly, we can define $\bm A_{,r:s}$ and $\bm x_{r:s}$ for a vector $\bm x$. By $\mathrm{vec}(\bm A)$, we mean the vector obtained by stacking the columns of the matrix $\bm A$ one over another. For an $m\times n$ matrix $\bm A$ and a $p\times q$ matrix $\bm B$, $\bm A\otimes\bm B$ denotes the Kronecker product between $\bm A$ and $\bm B$; see \citet{steeb2006problems} for the definition. An identity matrix of order $p$ is denoted by $\bm I_p$.
For a vector $\bm x\in\mathbb{R}^p$, we denote $\|\bm x\|=\left(\sum_{i=1}^px_i^2\right)^{1/2}$. We denote the $r^{th}$ order derivative of a function $f(\cdot)$ by $f^{(r)}(\cdot)$, that is, $f^{(r)}(t)=\frac{d^r}{dt^r}f(t)$. The boldfaced symbol $\bm f(\cdot)$ stands for a vector valued function. For functions $\bm f:[0,1]\rightarrow\mathbb{R}^p$ and $w:[0,1]\rightarrow[0,\infty)$, we define $\|\bm f\|_w=(\int_0^1 \|\bm f(t)\|^2w(t)dt)^{1/2}$. For a real-valued function $f:[0,1]\rightarrow\mathbb{R}$ and a vector $\bm x\in\mathbb{R}^p$, we denote $f(\bm x)=(f(x_1),\ldots,f(x_p))^T$. The notation ``:=" means equality by definition. The notation $\left<\cdot,\cdot\right>$ stands for inner product. For numerical sequences $r_n$ and $s_n$, by $r_n\asymp s_n$ we mean $r_n=O(s_n)$ and $s_n=O(r_n)$, while $r_n\lesssim s_n$ stands for $r_n=O(s_n)$. We use the notation $a\wedge b$ to imply $\min(a,b)$, $a, b$ being two real numbers. The boldfaced symbols $\bm\E(\cdot)$ and $\bm\Var(\cdot)$ stand for the mean vector and dispersion matrix respectively of a random vector. For the probability measures $P$ and $Q$ defined on $\mathbb{R}^p$, we define the total variation distance $\|P-Q\|_{TV}=\sup_{B\in\mathscr{R}^p}|P(B)-Q(B)|$, where $\mathscr{R}^p$ denotes the Borel $\sigma$-field on $\mathbb R^p$. For an open set $E$, the symbol $C^m(E)$ stands for the collection of functions defined on $E$ with first $m$ continuous partial derivatives with respect to its arguments. Now let us consider the formal description of the model.\\
\indent We have a system of $d$ ordinary differential equations given by
\begin{eqnarray}
\frac{df_{j\bm\theta}(t)}{dt}=F_j(t,\bm f_{\bm\theta}(t),\bm\theta),\,t\in[0,1],\,j=1,\ldots,d,\label{diff}
\end{eqnarray}
where $\bm f_{\bm\theta}(\cdot)=(f_{1\bm\theta}(\cdot),\ldots,f_{d\bm\theta}(\cdot))^T$ and $\bm\theta\in\bm\Theta$, where we assume that $\bm\Theta$ is a compact subset of $\mathbb{R}^p$. We also assume that for a fixed $\bm\theta$, $\bm F\in C^{m-1}((0,1),\mathbb{R}^d)$ for some integer $m\geq 1$. Then, by sucessive differentiation of the right hand side of \eqref{diff}, it follows that $\bm f_{\bm\theta}\in C^m((0,1))$. By the implied uniform continuity, the function and its several derivatives uniquely extend to continuous functions on $[0,1]$.\\
\indent Consider an $n\times d$ matrix of observations $\bm Y$ with $Y_{i,j}$ denoting the measurement taken on the $j^{th}$ response at the point $x_i,\,0\leq x_i\leq 1,\, i=1,\ldots,n;\,j=1,\ldots,d$. Denoting $\bm\varepsilon$ as the corresponding matrix of errors, the proposed model is given by
\begin{eqnarray}
Y_{i,j}=f_{j\bm\theta}(x_i)+\varepsilon_{i,j},\,i=1,\ldots,n,\,j=1,\ldots,d,\label{prop}
\end{eqnarray}
whereas the data is generated by the model
\begin{eqnarray}
Y_{i,j}=f_{j0}(x_i)+\varepsilon_{i,j},\,i=1,\ldots,n,\,j=1,\ldots,d,\label{true}
\end{eqnarray}
where $\bm f_0(\cdot)=(f_{10}(\cdot),\ldots,f_{d0}(\cdot))^T$ denotes the true mean which does not necessarily lie in $\{\bm f_{\bm\theta}:\bm\theta\in\bm\Theta\}$. Let $\varepsilon_{i,j}\stackrel{iid}\sim P_0$, which is a probability distribution with mean zero and finite variance $\sigma_0^2$ for $i=1,\ldots,n\,;j=1,\ldots,d$.\\
Since the expression of $\bm f_{\bm\theta}$ is usually not available, the proposed model is embedded in nonparametric regression model
\begin{eqnarray}
\bm Y=\bm X_n\bm B_n+\bm\varepsilon,\label{np}
\end{eqnarray}
where $\bm X_n=(\!(N_{j}(x_i))\!)_{1\leq i\leq n,1\leq j\leq k_n+m-1}$, $\{N_{j}(\cdot)\}_{j=1}^{k_n+m-1}$ being the B-spline basis functions of order $m$ with $k_n-1$ interior knots. Here we denote $\bm B_n=\left(\bm\beta_1^{(k_n+m-1)\times 1},\ldots,\bm\beta_d^{(k_n+m-1)\times 1}\right)$, the matrix containing the coefficients of the basis functions. Also we consider $P_0$ to be unknown and use $N(0,\sigma^2)$ as the working distribution for the error where $\sigma$ may be treated as another unknown parameter. Let us denote by $t_1,\ldots,t_{k_n-1}$ the set of interior knots with $t_l=l/{k_n}$ for $l=1,\ldots,k_n-1$. Hence the meshwidth is $\xi_n={1}/{k_n}$. Denoting by $Q_n$, the empirical distribution function of $x_i,\,i=1,\ldots,n$, we assume
\begin{eqnarray}
\sup_{t\in[0,1]}|Q_n(t)-t|=o(k_n^{-1}).\nonumber
\end{eqnarray}We induce a prior on $\bm\Theta$ from those assigned on $\bm\beta_j$, $j=1,\ldots,d$, by expanding the definition of the parameter beyond the assumed parametric model. We assume
\begin{equation}
\bm\beta_j\stackrel{iid}\sim N_{k_n+m-1}(\bm 0,nk_n^{-1}(\bm X^T_n\bm X_n)^{-1}).\label{prior}
\end{equation}
Simple calculation yields the posterior distribution for $\bm\beta_j$ as
\begin{eqnarray}
\bm\beta_j|\bm Y\sim N_{k_n+m-1}\left({\left(1+\frac{\sigma^2}{nk_n^{-1}}\right)}^{-1}{({\bm X^T_n\bm X_n})}^{-1}{\bm X^T_n\bm Y_{,j}},{\left(\frac{1}{\sigma^2}+\frac{1}{nk_n^{-1}}\right)}^{-1}(\bm X^T_n\bm X_n)^{-1}\right)\nonumber\\\label{posterior}
\end{eqnarray}
and the posterior distributions of $\bm\beta_j$ and $\bm\beta_{j'}$ are mutually independent for $j\neq j';\,j,j'=1,\ldots,d$. By model \eqref{np}, the expected response vector at a point $t\in[0,1]$ is given by $\bm B^T_n\bm N(t)$, where $\bm N(\cdot)=(N_{1}(\cdot),\ldots,N_{k_n+m-1}(\cdot))^T$.\\
\indent Let $w(\cdot)$ be a weight function with $w(0)=w(1)=0$. Denoting $\bm F(\cdot,\cdot,\cdot)=(F_1(\cdot,\cdot,\cdot),\ldots,F_d(\cdot,\cdot,\cdot))^T$, we define
\begin{eqnarray}
R_{\bm f}(\bm\eta)&=&\left\{\int_0^1\|\bm f'(t)-\bm F(t,\bm f(t),\bm\eta)\|^2w(t)dt\right\}^{1/2},\nonumber\\
\bm\psi(\bm f)&=&\arg\min_{\bm\eta\in\bm\Theta}R_{\bm f}(\bm\eta)\nonumber.
\end{eqnarray}
It is easy to check that $\bm\psi(\bm f_{\bm\eta})=\bm\eta$ for all $\bm\eta\in\Theta$. Thus the map $\bm\psi$ extends the definition of the parameter $\bm\theta$ beyond the model. Let us define $\bm\theta_0=\bm\psi(\bm f_0)$.
A prior is induced on $\bm\Theta$ through the mapping $\bm\psi$ acting on $\bm f(\cdot)=\bm B^T_n\bm N(\cdot)$ and the prior on $\bm B_n$ given by equation \eqref{prior}. From now on, we shall write $\bm\theta$ for $\bm\psi(\bm f)$ and treat it as the parameter of interest.
\section{Main results}
Our objective is to study the asymptotic behavior of the posterior distribution of $\sqrt{n}(\bm\theta-\bm\theta_0)$. The asymptotic representation of $\sqrt{n}(\bm\theta-\bm\theta_0)$ is given by the next theorem under the assumption that
\begin{equation}
\text{for all}\,\epsilon>0,\,\,\inf_{\bm\eta:\|\bm\eta-\bm\theta_0\|\geq\epsilon}R_{\bm f_0}(\bm\eta)>R_{\bm f_0}(\bm\theta_0).\label{assmp}
\end{equation}
We denote $D_{l,r,s}\bm F(t,\bm f,\bm\theta)={\partial^{l+r+s}}/{\partial\bm\theta^s\partial\bm f^r\partial t^l}\bm F(t,\bm f(t),\bm\theta)$. Since the posterior distributions of $\bm\beta_j$ are mutually independent when $\bm\varepsilon_{,j}$ are mutually independent for $j=1,\ldots,d$, we can assume $d=1$ in Theorem 1 for the sake of simplicity in notation and write $f(\cdot)$, $f_0(\cdot)$, $F(\cdot,\cdot,\cdot)$, $\bm\beta$ instead of $\bm f(\cdot)$, $\bm f_0(\cdot)$, $\bm F(\cdot,\cdot,\cdot)$ and $\bm B_n$ respectively. Extension to $d$-dimensional case is straightforward as shown in Remark 2 after the statement of Theorem 1. We deal with the situation of correlated errors in Section 5.
\begin{theorem}
Let the matrix
\begin{eqnarray}
\bm J_{\bm\theta_0}&=&\int_0^1(D_{0,0,1} F(t, f_0(t),\bm\theta_0))^TD_{0,0,1} F(t, f_0(t),\bm\theta_0)w(t)dt\nonumber\\
&&-\int_0^1\left(D_{0,0,1}\bm S(t, f_0(t),\bm\theta_0)\right)w(t)dt\nonumber
\end{eqnarray}
be nonsingular, where
\begin{eqnarray}
\bm S(t, f(t),\bm\theta)=(D_{0,0,1} F(t, f(t),\bm\theta))^T(f'_0(t)-F(t, f_0(t),\bm\theta_0)).\nonumber
\end{eqnarray}
Let $m$ and $k_n$ be chosen so that $k_n^{-1}=o(1)$, $\sqrt{n}k_n^{-m}=o(1)$ and $n^{1/2}k_n^{-4}\rightarrow\infty$ as $n\rightarrow\infty$. If $D_{0,1,1}F(t, y,\bm\theta)$, $D_{0,2,1}F(t, y,\bm\theta)$ and $D_{0,0,2}F(t, y,\bm\theta)$ are continuous in its arguments, then under the assumption \eqref{assmp}, there exists $E_n\subseteq\bm\Theta\times C^m((0,1))$ with $\Pi(E^c_n|\bm Y)=o_{P_0}(1)$, such that for $(\bm\theta, f)\in E_n$,
\begin{eqnarray}
\|\sqrt{n}(\bm\theta-\bm\theta_0)-\bm{J}_{\bm\theta_0}^{-1}\sqrt{n}(\bm\Gamma(f)-\bm\Gamma(f_{0}))\|&\lesssim&\sqrt{n}\sup_{t\in[0,1]}\|f(t)-f_0(t)\|^2\nonumber\\
&&+\sqrt{n}\sup_{t\in[0,1]}\|f'(t)-f'_0(t)\|^2,\label{thm1}
\end{eqnarray}
where
\begin{align}
\bm\Gamma(z)=&\int_0^1 \left(-(D_{0,0,1} F(t, f_0(t),\bm\theta_0))^TD_{0,1,0} F(t, f_0(t),\bm\theta_0)w(t)\right.\nonumber\\
&\left.-\frac{d}{dt}[(D_{0,0,1} F(t, f_0(t),\bm\theta_0))^Tw(t)]+\left(D_{0,1,0}\bm S(t, f_0(t),\bm\theta_0)\right)w(t)\right)z(t)dt.\nonumber
\end{align}
\end{theorem}
\subsection*{Remark 1}
Condition \eqref{assmp} implies that $\bm\theta_0$ is the unique point of minimum of $R_{\bm f_0}(\cdot)$ and $\bm\theta_0$ should be a well-separated point of minimum.
\subsection*{Remark 2}
When the response is a $d$-dimensional vector, \eqref{thm1} holds with the scalars being replaced by the corresponding $d$-dimensional vectors. Let us denote
\begin{eqnarray}
\bm A(t)&=&\bm J_{\bm\theta_0}^{-1}\{-(D_{0,0,1}\bm F(t,\bm f_0(t),\bm\theta_0))^TD_{0,1,0}\bm F(t,\bm f_0(t),\bm\theta_0)w(t)\nonumber\\
&&-\frac{d}{dt}[(D_{0,0,1}\bm F(t,\bm f_0(t),\bm\theta_0))^Tw(t)]\nonumber\\
&&+\left(D_{0,1,0}\bm S(t,\bm f_0(t),\bm\theta_0)\right)w(t)\},\nonumber
\end{eqnarray}
which is a $p\times d$ matrix. Then we have
\begin{eqnarray}
\bm J_{\bm\theta_0}^{-1}\bm\Gamma(\bm f)&=&\int_0^1\bm A(t)(\bm\beta_1,\ldots,\bm\beta_d)^T\bm N(t)dt\nonumber\\
&=&\sum_{j=1}^d\int_0^1\bm A_{,j}(t)\bm N^T(t)\bm\beta_jdt\nonumber\\
&=&\sum_{j=1}^d\bm G_{n,j}^T\bm\beta_j,\label{linearize}
\end{eqnarray}
where $\bm G_{n,j}^T=\int_0^1\bm A_{,j}(t)\bm N^T(t)dt$ which is a matrix of order $p\times (k_n+m-1)$ for $j=1,\ldots,d$. \citet{bontemps2011bernstein} gave a Bernstein-von Mises theorem for the posterior distribution of a parameter vector $\bm\beta_j$ under the working model $\bm Y_{,j}=\bm X_n\bm\beta_j+\bm\varepsilon_{,j}$, which can be repeatedly applied for $j=1,\ldots,d$. Then in order to approximate the posterior distribution of $\bm\theta$, it suffices to study the asymptotic posterior distribution of the linear combination of $\bm\beta_j$ given by \eqref{linearize}. The next theorem, which is a simple consequence of Theorem 1 above and Theorem 1 and Corollary 1 of \citet{bontemps2011bernstein}, describes the approximate posterior distribution of $\sqrt{n}(\bm\theta-\bm\theta_0)$.
\begin{theorem}

Let us denote
\begin{eqnarray}
\bm\mu_n&=&\sqrt{n}\sum_{j=1}^d\bm G_{n,j}^T(\bm X^T_n\bm X_n)^{-1}\bm X^T_n\bm Y_{,j}-\sqrt{n}\bm J_{\bm\theta_0}^{-1}\bm\Gamma(\bm f_0),\nonumber\\
\bm\Sigma_n&=&n\sum_{j=1}^d\bm G_{n,j}^T(\bm X^T_n\bm X_n)^{-1}\bm G_{n,j}\nonumber
\end{eqnarray}
and $\bm B_j=\left(\!\left(\left<A_{k,j}(\cdot),A_{k',j}(\cdot)\right>\right)\!\right)_{k,k'=1,\ldots,p}$ for $j=1,\ldots,d$. If $\bm B_j$ is non-singular for all $j=1,\ldots,d$, then under the conditions of Theorem 1,
\begin{eqnarray}
\left\|\Pi\left(\sqrt{n}(\bm\theta-\bm\theta_0)\in\cdot|\bm Y\right)-N\left(\bm\mu_n,\sigma^2\bm\Sigma_n\right)\right\|_{TV}=o_{P_0}(1).\label{thm2}
\end{eqnarray}
\end{theorem}
\subsection*{Remark 3}
Inspecting the proof, we can conclude that \eqref{thm2} is uniform over $\sigma^2$ belonging to a compact subset of $(0,\infty)$. Also note that the scale of the approximating normal distribution involves the working variance $\sigma^2$ assuming that it is given, even though the convergence is studied under the true distribution $P_0$ with variance $\sigma_0^2$, not necessarily equal to the given $\sigma^2$. Thus, the distribution matches with the frequentist distribution of the estimator in \citet{brunel2008parameter} only if $\sigma$ is correctly specified as $\sigma_0$. Below we assure that putting a prior on $\sigma$ rectifies the problem.\\
\indent We assign independent $N(\bm 0, nk^{-1}_n\sigma^2(\bm X_n^TX_n)^{-1})$ prior on $\bm\beta_j$ for $j=1,\ldots,d$, and inverse gamma prior on $\sigma^2$ with parameters $a$ and $b$. Then, the marginal posterior of $\sigma^2$ is also inverse gamma with parameters $(d(n-k_n-m+1)+2a)/2$ and $b+\sum_{j=1}^d\bm Y^T_{,j}(\bm I_n-\bm P_{\bm X_n}(1+(k_n/n))^{-1})\bm Y_{,j}/2$, where $\bm P_{\bm X_n}=\bm X_n(\bm X^T_n\bm X_n)^{-1}\bm X^T_n$. Let us assume that the fourth order moment of the true error distribution is finite. Straightforward calculations show that $|\mathrm E(\sigma^2|\bm Y)-\sigma_0^2|=O_{P_{0}}(n^{-1/2})$ and $\mathrm Var(\sigma^2|\bm Y)=O_{P_{0}}(n^{-1})$. In particular, the marginal posterior distribution of $\sigma^2$ is consistent at the true value of error variance. Let $\mathscr{N}$ be an arbitrary neighborhood of $\sigma_0$. Then, $\Pi(\mathscr{N}^c|\bm Y)=o_{P_0}(1)$. We observe that
\begin{eqnarray}
\lefteqn{\sup_{B\in\mathscr{R}^p}\left|\Pi(\sqrt{n}(\bm\theta-\bm\theta_0)\in B|\bm Y)-\Phi(B;\bm\mu_n,\sigma_0^2\bm\Sigma_n)\right|}\nonumber\\
&\leq&\int\sup_{B\in\mathscr{R}^p}\left|\Pi(\sqrt{n}(\bm\theta-\bm\theta_0)\in B|\bm Y,\sigma)-\Phi(B;\bm\mu_n,\sigma^2\bm\Sigma_n)\right|d\Pi(\sigma|\bm Y)\nonumber\\
&&+\int\sup_{B\in\mathscr{R}^p}\left|\Phi(B;\bm\mu_n,\sigma^2\bm\Sigma_n)-\Phi(B;\bm\mu_n,\sigma_0^2\bm\Sigma_n)\right|d\Pi(\sigma|\bm Y)\nonumber\\
&\leq&\sup_{\sigma\in\mathscr{N}}\sup_{B\in\mathscr{R}^p}\left|\Pi(\sqrt{n}(\bm\theta-\bm\theta_0)\in B|\bm Y,\sigma)-\Phi(B;\bm\mu_n,\sigma^2\bm\Sigma_n)\right|\nonumber\\
&&+\int_{\mathscr{N}}\sup_{B\in\mathscr{R}^p}\left|\Phi(B;\bm\mu_n,\sigma^2\bm\Sigma_n)-\Phi(B;\bm\mu_n,\sigma_0^2\bm\Sigma_n)\right|d\Pi(\sigma|\bm Y)+2\Pi(\mathscr{N}^c|\bm Y).\nonumber
\end{eqnarray}
The second integral is bounded by $\int_{\mathscr{N}}(C_{\sigma_0}|\sigma-\sigma_0|\wedge1)d\Pi(\sigma|\bm Y)$, $C_{\sigma_0}$ being a constant depending only on $\sigma_0$. This bound can be made arbitrarily small by choosing $\mathscr{N}$ sufficiently small. The first integral converges in probability to zero by \eqref{thm2}. The third term converges in probability to zero by the posterior consistency. Hence, we get the desired result. This result can be extended for the situation of correlated errors using a little more manipulation.
\section{Extension}
The results obtained so far can be extended for the case where $\bm\varepsilon_{i,j}$ and $\bm\varepsilon_{i,j'}$ are associated for $i=1,\ldots,n$ and $j\neq j';\,j,j'=1,\ldots,d$. Let under the working model, $\bm\varepsilon_{i,}$ have the dispersion matrix $\bm\Sigma^{d\times d}=\left(\!\left(\sigma_{jk}\right)\!\right)_{j,k=1}^d$ for $i=1,\ldots,n$. Denoting $\bm\Sigma^{-1/2}=(\!(\sigma^{jk})\!)_{j,k=1}^d$, we have the following extension of Theorem 2.
\begin{theorem}

Let us denote
\begin{eqnarray}
\bm\mu^*_n&=&\sqrt{n}\sum_{k=1}^d\left[\left(\bm G^T_{n,1}\ldots\bm G^T_{n,d}\right)\left(\bm\Sigma^{1/2}\otimes\bm I_{k_n+m-1}\right)\right]_{,_{(k-1)(k_n+m-1)+1:k(k_n+m-1)}}\nonumber\\
&&\times{(\bm X^T_n\bm X_n)}^{-1}\bm X_n^T\sum_{j=1}^d\bm Y_j\sigma^{jk}
-\sqrt{n}\bm J_{\bm\theta_0}^{-1}\bm\Gamma(\bm f_0),\nonumber\\
\bm\Sigma^*_n&=&n\sum_{k=1}^d\left[\left(\bm G^T_{n,1}\ldots\bm G^T_{n,d}\right)\left(\bm\Sigma^{1/2}\otimes\bm I_{k_n+m-1}\right)\right]_{,{(k-1)(k_n+m-1)+1:k(k_n+m-1)}}\nonumber\\
&&\times{(\bm X^T_n\bm X_n)}^{-1}\nonumber\\
&&\times\left[\left(\bm G^T_{n,1}\ldots\bm G^T_{n,d}\right)\left(\bm\Sigma^{1/2}\otimes\bm I_{k_n+m-1}\right)\right]^T_{{(k-1)(k_n+m-1)+1:k(k_n+m-1)},}.\nonumber
\end{eqnarray}
Then under the conditions of Theorem 1 and Theorem 2,
\begin{eqnarray}
\left\|\Pi\left(\sqrt{n}(\bm\theta-\bm\theta_0)\in\cdot|\bm Y\right)-N\left(\bm\mu^*_n,\bm\Sigma^*_n\right)\right\|_{TV}=o_{P_0}(1).\label{thm3}
\end{eqnarray}
\end{theorem}
\subsection*{Remark 4}
In many applications, the regression function is modeled as $\bm h_{\bm\theta}(t)=\bm g(\bm f_{\bm\theta}(t))$ instead of $\bm f_{\bm\theta}(t)$, where $\bm g$ is a known invertible function and $\bm h_{\bm\theta}(t)\in\mathbb{R}^d$. It should be noted that
\begin{eqnarray}
\frac{d\bm h_{\bm\theta}(t)}{dt}=\bm g'(\bm f_{\bm\theta}(t))\frac{d\bm f_{\bm\theta}(t)}{dt}&=&\bm g'(\bm g^{-1}\bm h_{\bm\theta}(t))\bm F(t,\bm g^{-1}\bm h_{\bm\theta}(t),\bm\theta)\nonumber\\
&=&\bm H(t,\bm h_{\bm\theta}(t),\bm\theta),\nonumber
\end{eqnarray}
which is known function of $t, \bm h_{\bm\theta}$ and $\bm\theta$. Now we can carry out our analysis replacing $\bm F$ and $\bm f_{\bm\theta}$ in \eqref{intro} by $\bm H$ and $\bm h_{\bm\theta}$ respectively.
\section{Simulation Study}
We consider two differential equation models to study the posterior distribution of $\bm\theta$. In each case we consider both the situations where the true regression function belongs to the solution set and lies outside the solution set. For a sample of size $n$, the $x_i$'s are chosen as $x_i=(2i-1)/2n$ for $i=1,\ldots,n$. We consider sample sizes 50, 100, 200, 500 and 1000. We simulate 1000 replications for each case. Under each replication a sample of size 1000 is directly drawn from the posterior distribution of $\bm\theta$ and then the 95\% credible interval is obtained. We calculate the coverage as well as the length of the corresponding credible interval. The average length of the credible interval is obtained across various replications. The estimated standard errors of the interval length and coverage are given inside the parentheses in the tables. The same procedure has been repeated to construct the 95$\%$ bootstrap percentile intervals. %The standard errors are given inside the parentheses in the tables.
\subsection*{Example 1}
The ODE is given by
\begin{eqnarray}
F(t,f(t),\theta)=\theta t-\theta tf(t),\,\,t\in[0,1],\nonumber
\end{eqnarray}
with initial condition $f_{\theta}(0)=2$.
The corresponding solution is
\begin{equation}
f_{\theta}(t)=1+\exp(-\theta t^2/2).\nonumber
\end{equation}
Now we consider two cases.
\subsection*{Case 1}
The true regression function is in the model. We take $\theta_0=1$ to be the true value of the parameter.
\subsection*{Case 2}
The true mean is taken outside the solution set of the ODE:
\begin{equation}
f_0(t)=1+\exp(-t^2/2)+0.02\sin(4\pi t).\nonumber
\end{equation}
The true distribution of error is taken either $N(0,1)$ or $t$-distribution with $6$ degrees of freedom. We put an inverse gamma prior on $\sigma^2$ with both parameters being $1$. The simulation results are summarized in the tables and 2 which show the superiority of the Bayesian approach.

\begin{table}[h]
\centering
\caption{\textit{Coverages and average lengths of the Bayesian credible interval and bootstrap percentile interval of $\theta$ for normally distributed errors}}
\begin{tabular}{|c|cc|cc|cc|cc|}
\hline
\multirow{2}{*}{$n$}&\multicolumn{4}{|c|}{Case 1}&\multicolumn{4}{|c|}{Case 2}\\
\cline{2-9}
&\multicolumn{2}{|c|}{Bayes}&\multicolumn{2}{|c|}{Bootstrap}&\multicolumn{2}{|c|}{Bayes}&\multicolumn{2}{|c|}{Bootstrap}\\
\hline
&coverage&length&coverage&length&coverage&length&coverage&length\\
&(se)&(se)&(se)&(se)&(se)&(se)&(se)&(se)\\
50&99.0&9.54&71.4&5.09&98.7&9.54&71.0&5.10\\
&(0.01)&(1.65)&(0.06)&(1.04)&(0.02)&(1.68)&(0.06)&(1.04)\\
100&97.5&4.80&86.8&3.24&97.6&4.79&86.8&3.24\\
&(0.02)&(0.90)&(0.03)&(0.44)&(0.02)&(0.90)&(0.03)&(0.44)\\
200&95.7&3.04&90.6&2.38&95.5&3.03&90.5&2.38\\
&(0.01)&(0.54)&(0.02)&(0.25)&(0.01)&(0.53)&(0.02)&(0.25)\\
500&97.1&1.77&95.7&1.58&96.5&1.76&95.7&1.58\\
&(0.01)&(0.22)&(0.01)&(0.12)&(0.01)&(0.22)&(0.01)&(0.12)\\
1000&95.9&1.20&96.4&1.13&96.0&1.20&96.5&1.14\\
&(0.01)&(0.12)&(0.01)&(0.07)&(0.01)&(0.12)&(0.01)&(0.07)\\
\hline
\end{tabular}
\end{table}

\begin{table}[h]
\centering
\caption{\textit{Coverages and average lengths of the Bayesian credible interval and bootstrap percentile interval of $\theta$ for errors having t distribution with 6 degrees of freedom}}
\begin{tabular}{|c|cc|cc|cc|cc|}
\hline
\multirow{2}{*}{$n$}&\multicolumn{4}{|c|}{Case 1}&\multicolumn{4}{|c|}{Case 2}\\
\cline{2-9}
&\multicolumn{2}{|c|}{Bayes}&\multicolumn{2}{|c|}{Bootstrap}&\multicolumn{2}{|c|}{Bayes}&\multicolumn{2}{|c|}{Bootstrap}\\
\hline
&coverage&length&coverage&length&coverage&length&coverage&length\\
&(se)&(se)&(se)&(se)&(se)&(se)&(se)&(se)\\
50&95.1&9.71&68.0&6.56&94.1&9.67&67.8&6.57\\
&(0.03)&(2.02)&(0.06)&(2.00)&(0.03)&(2.01)&(0.06)&(2.00)\\
100&94.4&5.06&84.8&3.98&94.7&5.04&84.8&3.98\\
&(0.02)&(1.07)&(0.04)&(0.81)&(0.02)&(1.06)&(0.04)&(0.81)\\
200&94.5&3.43&90.0&2.89&94.4&3.42&90.1&2.89\\
&(0.02)&(0.70)&(0.02)&(0.42)&(0.02)&(0.68)&(0.02)&(0.42)\\
500&93.5&2.07&91.8&1.90&93.2&2.05&91.8&1.90\\
&(0.01)&(0.32)&(0.01)&(0.18)&(0.01)&(0.32)&(0.01)&(0.18)\\
1000&94.9&1.44&94.8&1.38&94.9&1.43&95.0&1.38\\
&(0.01)&(0.17)&(0.01)&(0.10)&(0.01)&(0.17)&(0.01)&(0.10)\\
\hline
\end{tabular}
\end{table}
\subsection*{Example 2}
We take $p=d=2$ and the ODE's are given by
\begin{eqnarray}
F_1(t, \bm f(t), \bm\theta)&=&\theta_1f_1(t),\nonumber\\
F_2(t, \bm f(t), \bm\theta)&=&2\theta_2f_1(t)+\theta_1f_2(t),\,\,t\in[0,1],\nonumber
\end{eqnarray}
with initial condition $f_{1\bm\theta}(0)=f_{2\bm\theta}(0)=1$.
The solution to the above system is given by
\begin{eqnarray}
f_{1\bm\theta}(t)&=&\exp(\theta_1t),\nonumber\\
f_{2\bm\theta}(t)&=&(2\theta_2t+1)\exp(\theta_1t).\nonumber
\end{eqnarray}
Again we consider two cases.
\subsection*{Case 1}
The true regression function is in the model. We take $\theta_{10}=\theta_{20}=1$ to be the true value of the parameter.
\subsection*{Case 2}
The true mean vector is taken outside the solution set of the ODE:
\begin{eqnarray}
f_{10}(t)&=&\exp(t)+0.1\sin(4\pi t),\nonumber\\
f_{20}(t)&=&(2t+1)\exp(t)+0.45\cos(4\pi t).\nonumber
\end{eqnarray}
For each case the true distribution of error is taken a standard normal distribution. Here we use $\sigma=1$, that is the true error variance. We obtained the $95\%$ credible and bootstrap percentile intervals for $\theta_1$ and $\theta_2$ separately and calculated the corresponding coverages and lengths.\\
\indent The simulation results are summarized in tables 3 and 4. The tables show that the Bayesian method works much better than the bootstrap technique across different sample sizes.
\begin{table}[h]
\centering
\caption{\textit{Coverages and average lengths of the Bayesian credible interval and bootstrap percentile interval to estimate $\theta_1$ for normally distributed errors}}
\begin{tabular}{|c|cc|cc|cc|cc|}
\hline
\multirow{2}{*}{$n$}&\multicolumn{4}{|c|}{Case 1}&\multicolumn{4}{|c|}{Case 2}\\
\cline{2-9}
&\multicolumn{2}{|c|}{Bayes}&\multicolumn{2}{|c|}{Bootstrap}&\multicolumn{2}{|c|}{Bayes}&\multicolumn{2}{|c|}{Bootstrap}\\
\hline
&coverage&length&coverage&length&coverage&length&coverage&length\\
&(se)&(se)&(se)&(se)&(se)&(se)&(se)&(se)\\
50&100.0&4.34&94.5&3.40&100.0&4.30&90.6&3.44\\
&(0.00)&(0.58)&(0.03)&(1.11)&(0.00)&(2.25)&(0.04)&(1.11)\\
100&99.9&3.20&97.6&2.28&99.9&3.17&96.7&2.33\\
&(0.00)&(0.43)&(0.02)&(0.40)&(0.00)&(0.42)&(0.02)&(0.45)\\
200&100.0&2.28&92.7&1.56&100.0&2.27&93.4&1.62\\
&(0.00)&(0.30)&(0.02)&(0.19)&(0.00)&(0.29)&(0.02)&(0.21)\\
500&99.8&1.18&85.2&0.79&99.7&1.20&74.5&0.87\\
&(0.00)&(0.14)&(0.02)&(0.08)&(0.00)&(0.13)&(0.02)&(0.09)\\
1000&99.9&0.73&85.0&0.49&99.6&0.76&64.6&0.55\\
&(0.00)&(0.08)&(0.01)&(0.04)&(0.00)&(0.07)&(0.02)&(0.05)\\
\hline
\end{tabular}
\end{table}

%We repeat the same procedure for errors having $t$-distribution with $20$ degrees of freedom. Again, $\theta_1$ is estimated in a more efficient manner using the Bayesian principle.
\begin{table}[h]
\centering
\caption{\textit{Coverages and average lengths of the Bayesian credible interval and bootstrap percentile interval to estimate $\theta_2$ for normally distributed errors}}
\begin{tabular}{|c|cc|cc|cc|cc|}
\hline
\multirow{2}{*}{$n$}&\multicolumn{4}{|c|}{Case 1}&\multicolumn{4}{|c|}{Case 2}\\
\cline{2-9}
&\multicolumn{2}{|c|}{Bayes}&\multicolumn{2}{|c|}{Bootstrap}&\multicolumn{2}{|c|}{Bayes}&\multicolumn{2}{|c|}{Bootstrap}\\
\hline
&coverage&length&coverage&length&coverage&length&coverage&length\\
&(se)&(se)&(se)&(se)&(se)&(se)&(se)&(se)\\
50&99.9&6.57&87.5&5.18&100.0&6.51&87.1&5.26\\
&(0.00)&(1.15)&(0.05)&(2.63)&(0.00)&(1.15)&(0.05)&(2.53)\\
100&100.0&4.61&87.4&3.22&99.9&4.57&87.5&3.31\\
&(0.00)&(0.75)&(0.03)&(0.78)&(0.00)&(0.74)&(0.03)&(0.93)\\
200&100.0&3.20&88.4&2.13&100.0&3.18&85.0&2.20\\
&(0.00)&(0.48)&(0.02)&(0.29)&(0.00)&(0.48)&(0.02)&(0.31)\\
500&99.7&1.61&90.4&1.06&99.7&1.64&81.5&1.15\\
&(0.00)&(0.21)&(0.01)&(0.11)&(0.00)&(0.20)&(0.02)&(0.12)\\
1000&99.9&0.98&92.6&0.64&99.3&1.03&76.3&0.71\\
&(0.00)&(0.11)&(0.01)&(0.05)&(0.00)&(0.11)&(0.01)&(0.06)\\
\hline
\end{tabular}
\end{table}
In both examples we observed that the two methods are roughly comparable in term of computational efficiency.
\section{Proofs}
We need to go through a couple of lemmas in order to prove the main results. We denote by $\E_0(\cdot)$ and ${\Var}_0(\cdot)$ the expectation and variance operators respectively with respect to $P_0$ probability. The following lemma helps to estimate the bias of the Bayes estimator.
\begin{lemma}
For $m$ and $k_n$ chosen so that $k_n=o(n)$ and $\sqrt{n}k_n^{-m}=o(1)$, we have
\begin{eqnarray}
\sqrt{n}k_n^{-r-1/2}\sup_{t\in[0,1]}|\E_0(\E(f^{(r)}(t)|\bm Y))-f^{(r)}_{0}(t)|=o(1)\nonumber
\end{eqnarray}
for $r=0,1$.
\end{lemma}
\begin{proof}
We note that $f^{(r)}(t)=(\bm N^{(r)}(t))^T\bm\beta$ for $r=0,1$ with $\bm N^{(r)}(\cdot)$ standing for the $r^{th}$ order derivative of $\bm N(\cdot)$. By \eqref{posterior},
\begin{equation}
\E(f^{(r)}(t)|\bm Y)={\left(1+\frac{k_n\sigma^2}{n}\right)}^{-1}(\bm N^{(r)}(t))^T{({\bm X^T_n\bm X_n})}^{-1}\bm X^T_n\bm Y.\nonumber
\end{equation}
\citet{zhou2000derivative} showed that
\begin{eqnarray}
\Var_0(\E(f^{(r)}(t)|\bm Y))&=&\sigma_0^2{\left(1+\frac{\sigma^2k_n}{n}\right)}^{-2}(\bm N^{(r)}(t))^T{({\bm X^T_n\bm X_n})}^{-1}\bm N^{(r)}(t)\nonumber\\
&\asymp&\frac{k_n^{2r+1}}{n}.\label{var}
\end{eqnarray}
There exists a $\bm\beta^*$ \citep{barrow1979efficientl} such that
\begin{equation}
\sup_{t\in[0,1]}|f^{(r)}_{0}(t)-(\bm N^{(r)}(t))^T\bm\beta^*|=O({k_n^{-(m-r)}}).\label{spldis}
\end{equation}
Now,
\begin{eqnarray}
\E_0(\E(f^{(r)}(t)|\bm{Y}))&=&
{\left(1+\frac{\sigma^2k_n}{n}\right)}^{-1}(\bm N^{(r)}(t))^T{({\bm X^T_n\bm X_n})}^{-1}\bm X^T_nf_{0}(\bm{x}).\nonumber
\end{eqnarray}
For a $\bm\beta^*$ as in \eqref{spldis}, we can write the bias of $\E(f^{(r)}(t)|\mathbf{Y})$ as
\begin{eqnarray}
&&{\left(1+\frac{\sigma^2k_n}{n}\right)}^{-1}(\bm N^{(r)}(t))^T\bm\beta^*-(\bm N^{(r)}(t))^T\bm\beta^*\nonumber\\
&&+{\left(1+\frac{\sigma^2k_n}{n}\right)}^{-1}(\bm N^{(r)}(t))^T{({\bm X^T_n\bm X_n})}^{-1}\bm X^T_n(f_{0}(\bm{x})-\bm X_n\bm\beta^*)\nonumber\\
&&-(f^{(r)}_{0}(t)-(\bm N^{(r)}(t))^T\bm\beta^*).\nonumber
\end{eqnarray}
We note that
\begin{eqnarray}
\lefteqn{\sqrt{n}k_n^{-r-1/2}\sup_{t\in[0,1]}|\E_0(\E(f^{(r)}(t)|\bm{Y}))-f^{(r)}_{0}(t)|}\nonumber\\
&\leq&\sqrt{n}k_n^{-r-1/2}\sup_{t\in[0,1]}\left|{\left(1+\frac{k_n\sigma^2}{n}\right)}^{-1}(\bm N^{(r)}(t))^T\bm\beta^*-(\bm N^{(r)}(t))^T\bm\beta^*\right|\nonumber\\
&&+\sqrt{n}k_n^{-r-1/2}{\left(1+\frac{k_n\sigma^2}{n}\right)}^{-1}\sup_{t\in[0,1]}|(\bm N^{(r)}(t))^T{({\bm X^T_n\bm X_n})}^{-1}\bm X^T_n(f_{0}(\bm{x})-\bm X_n\bm\beta^*)|\nonumber\\
&&+\sqrt{n}k_n^{-r-1/2}\sup_{t\in[0,1]}|f^{(r)}_{0}(t)-(\bm N^{(r)}(t))^T\bm\beta^*|.\nonumber
\end{eqnarray}
The first term on the right hand side of the above is of the order of $k_n^{-r+1/2}/\sqrt{n}$. Using the Cauchy-Schwarz inequality, \eqref{var} and \eqref{spldis}, we can bound the second term up to a constant by $\sqrt{n}k_n^{-m}$. The third term has the order of $\sqrt{n}k_n^{-m-1/2}$ as a result of \eqref{spldis}.
By the conditions imposed on $m$ and $k_n$, we can finally conclude
\begin{equation}
\sqrt{n}k_n^{-r-1/2}\sup_{t\in[0,1]}|\E_0(\E(f^{(r)}(t)|\bm Y))-f^{(r)}_{0}(t)|=o(1).\nonumber
\end{equation}
\end{proof}
The following lemma controls posterior variability.
\begin{lemma}
Let $m$ and $k_n$ be chosen so that $k^{-1}_n=o(1)$, $\sqrt{n}k_n^{-m}=o(1)$ and $n^{1/2}k_n^{-4}\rightarrow\infty$ as $n\rightarrow\infty$. Then for all $\epsilon>0$,
\begin{eqnarray}
\E_0\Pi\left(\sqrt{n}\sup_{t\in[0,1]}(f^{(r)}(t)-f^{(r)}_{0}(t))^2>\epsilon|\bm Y\right)=o(1)\nonumber
\end{eqnarray}
for $r=0,1$.
\end{lemma}
\begin{proof}
By Markov's inequality and the fact that $(a+b)^2\leq 2(a^2+b^2)$ for two real numbers $a$ and $b$, we get
\begin{eqnarray}
\lefteqn{\Pi\left(\sup_{t\in[0,1]}\sqrt{n}(f^{(r)}(t)-f^{(r)}_{0}(t))^2>\epsilon|\bm Y\right)}\nonumber\\
&\leq&\epsilon^{-1}\E\left(\sup_{t\in[0,1]}\sqrt{n}(f^{(r)}(t)-f^{(r)}_{0}(t))^2|\bm Y\right)\nonumber\\
&\leq&{2\sqrt{n}}{\epsilon^{-1}}\left\{\sup_{t\in[0,1]}(\E(f^{(r)}(t)|\bm Y)-f^{(r)}_{0}(t))^2\right.\nonumber\\
&&\left.+\E\left[\sup_{t\in[0,1]}(f^{(r)}(t)-\E(f^{(r)}(t)|\bm Y))^2|\bm Y\right]\right\}.\label{chv}
\end{eqnarray}
Now we obtain the asymptotic orders of the expectations of the two terms inside the bracket above. We can write
\begin{eqnarray}
\E_0\sup_{t\in[0,1]}(\E(f^{(r)}(t)|\bm Y)-f^{(r)}_{0}(t))^2&\leq&2\sup_{t\in[0,1]}(\E_0(\E(f^{(r)}(t)|\bm Y))-f^{(r)}_{0}(t))^2\nonumber\\
&&+2\E_0\left[\sup_{t\in[0,1]}(\E(f^{(r)}(t)|\bm Y)-\E_0(\E[f^{(r)}(t)|\bm Y]))^2\right].\label{chv_1}
\end{eqnarray}
We have that
\begin{eqnarray}
A_1&:=&\sup_{t\in[0,1]}\left|\E(f^{(r)}(t)|\bm Y)-\E_0(\E[f^{(r)}(t)|\bm Y])\right|\nonumber\\
&=&\sup_{t\in[0,1]}{\left(1+\frac{\sigma^2k_n}{n}\right)}^{-1}\left|(\bm N^{(r)}(t))^T{({\bm X^T_n\bm X_n})}^{-1}\bm X^T_n(\bm Y-f_{0}(\bm x))\right|\nonumber\\
&=&\sup_{t\in[0,1]}{\left(1+\frac{\sigma^2k_n}{n}\right)}^{-1}\left|(\bm N^{(r)}(t))^T{({\bm X^T_n\bm X_n})}^{-1}\bm X^T_n\bm\varepsilon\right|\nonumber\\
&\leq&{\left(1+\frac{\sigma^2k_n}{n}\right)}^{-1}\left(\max_{1\leq k\leq n}\left|(\bm N^{(r)}(s_k))^T{({\bm X^T_n\bm X_n})}^{-1}\bm X^T_n\bm\varepsilon\right|\right.\nonumber\\
&&\left.+\sup_{t,t':|t-t'|\leq n^{-1}}\left|(\bm N^{(r)}(t)-\bm N^{(r)}(t'))^T{({\bm X^T_n\bm X_n})}^{-1}\bm X^T_n\bm\varepsilon\right|\right)\nonumber,
\end{eqnarray}
where $s_k=k/n$ for $k=1,\ldots,n$. Applying the mean value theorem to the second term of the above sum, we get
\begin{eqnarray}
A_1&\leq&{\left(1+\frac{k_n\sigma^2}{n}\right)}^{-1}\left(\max_{1\leq k\leq n}\left|(\bm N^{(r)}(s_k))^T{({\bm X^T_n\bm X_n})}^{-1}\bm X^T_n\bm \varepsilon\right|\right.\nonumber\\
&&\left.+\sup_{t\in[0,1]}\frac{1}{n}\left|(\bm N^{(r+1)}(t))^T{({\bm X^T_n\bm X_n})}^{-1}\bm X^T_n\bm\varepsilon\right|\right),\nonumber
\end{eqnarray}
and hence,
\begin{eqnarray}
A_1^2&\lesssim&\max_{1\leq k\leq n}\left|(\bm N^{(r)}(s_k))^T{({\bm X^T_n\bm X_n})}^{-1}\bm X^T_n\bm\varepsilon\right|^2\nonumber\\
&&+\sup_{t\in[0,1]}\frac{1}{n^2}\left|(\bm N^{(r+1)}(t))^T{({\bm X^T_n\bm X_n})}^{-1}\bm X^T_n\bm\varepsilon\right|^2.\label{A1}
\end{eqnarray}
By the spectral decomposition, we can write $\bm X_n{({\bm X^T_n\bm X_n})}^{-1}\bm X^T_n=\bm P^T\bm D\bm P$, where $\bm P$ is an orthogonal matrix and $\bm D$ is a diagonal matrix with $k_n+m-1$ ones and $n-k_n-m+1$ zeros in the diagonal. Now using the Cauchy-Schwarz inequality, we get
\begin{eqnarray}
\E_0\left(\max_{1\leq k\leq n}\left|(\bm N^{(r)}(s_k))^T{({\bm X^T_n\bm X_n})}^{-1}\bm X^T_n\bm\varepsilon\right|^2\right)&\leq&\max_{1\leq k\leq n}(\bm N^{(r)}(s_k))^T{({\bm X^T_n\bm X_n})}^{-1}\bm N^{(r)}(s_k)\nonumber\\
&&\times\E_0\left(\bm\varepsilon^T\bm P^T\bm D\bm P\bm\varepsilon\right).\nonumber
\end{eqnarray}
By Lemma 5.4 of \citet{zhou2000derivative} and the fact that $\bm\Var_0(\bm P\bm\varepsilon)=\bm\Var_0(\bm\varepsilon)$, we can conclude that the expectation of the first term on the right hand side of \eqref{A1} is $O((k_n^{2r+1}/n)k_n)$. Again applying the Cauchy-Schwarz inequality, the second term on the right hand side of \eqref{A1} is bounded by $\sup_{t\in[0,1]}\frac{1}{n^2}(\bm N^{(r+1)}(t))^T{(\bm X_n^T\bm X_n)}^{-1} \bm N^{(r+1)}(t)(\bm\varepsilon^T\bm\varepsilon)$, whose expectation is of the order $n(k_n^{2r+3}/n^3)=k_n^{2r+3}/n^2$, using Lemma 5.4 of Zhou and Wolfe (2000).
Thus, $\E_0(A_1^2)=O\left(k_n^{2r+2}/n\right)$. In Lemma 1, we have already proved that
\begin{eqnarray}
\sup_{t\in[0,1]}\left|\E_0(\E(f^{(r)}(t)|\bm{Y}))-f^{(r)}_{0}(t)\right|^2=O\left(k_n^{2r+1}/n\right),\nonumber
\end{eqnarray}
which together with \eqref{chv_1} gives
\begin{equation}
\E_0\left[\sup_{t\in[0,1]}\left|\E(f^{(r)}(t)|\bm Y)-f^{(r)}_{0}(t)\right|^2\right]=O\left(\frac{k_n^{2r+2}}{n}\right).\label{chv_2}
\end{equation}
Now we consider the second term inside the bracket of \eqref{chv}. Let us denote $\bm \varepsilon^*={(\bm X^T_n\bm X_n)}^{1/2}\bm\beta-{\left(1+\frac{\sigma^2k_n}{n}\right)}^{-1}{(\bm X^T_n\bm X_n)}^{-1/2}\bm X^T_n\bm Y$. It should be noted that $\bm\varepsilon^*|\bm Y\sim N(\bm 0, {\left({\sigma^{-2}}+{k_n/n}\right)}^{-1}\bm I_{k_n+m-1})$. Then we have
\begin{eqnarray}
A_2&:=&\sup_{t\in[0,1]}|f^{(r)}(t)-\E[f^{(r)}(t)|\bm Y]|\nonumber\\
&=&\sup_{t\in[0,1]}\left|(\bm N^{(r)}(t))^T\bm\beta-{\left(1+\frac{k_n\sigma^2}{n}\right)}^{-1}(\bm N^{(r)}(t))^T{({\bm X^T_n\bm X_n})}^{-1}\bm X^T_n\bm Y\right|\nonumber\\
&=&\sup_{t\in[0,1]}\left|(\bm N^{(r)}(t))^T{({\bm X^T_n\bm X_n})}^{-1/2}\nonumber\right.\\
&&\left.\times\left[{({\bm X^T_n\bm X_n})}^{1/2}\bm\beta-{\left(1+\frac{\sigma^2k_n}{n}\right)}^{-1}{({\bm X^T_n\bm X_n})}^{-1/2}\bm X^T_n\bm Y\right]\right|\nonumber\\
&=&\sup_{t\in[0,1]}\left|(\bm N^{(r)}(t))^T{({\bm X^T_n\bm X_n})}^{-1/2}\bm\varepsilon^*\right|.\nonumber
\end{eqnarray}
Now using the Cauchy-Schwarz inequality and Lemma 5.4 of \citet{zhou2000derivative}, we get $\E(A_2^2|\bm Y)=O\left(k_n^{2r+1}n^{-1}k_n\right)=O(k_n^{2r+2}/n)$. Combining it with \eqref{chv} and \eqref{chv_2} and utilizing the criteria on $n$ and $k_n$, we have
\begin{eqnarray}
\E_0\Pi(\sup_{t\in[0,1]}\sqrt{n}(f^{(r)}(t)-f^{(r)}_{0}(t))^2>\epsilon|\bm Y)=O\left({n^{-1/2}k_n^{2r+2}}\right)=o(1)\nonumber
\end{eqnarray}
for $r=0,1$. \end{proof}
In order to satisfy the criteria of Lemma 2, we must have $nk_n^{-8}\rightarrow\infty$ and $nk_n^{-2m}\rightarrow0$ as $n\rightarrow\infty$. Thus, $m\geq5$.
Lemmas 1 and 2 can be used to prove the posterior consistency of $\bm\theta$ as shown in the next lemma.
\begin{lemma}
If $m$ and $k_n$ satisfy the criteria of Lemmas 1 and 2, then for all $\epsilon>0$, $\Pi(\|\bm\theta-\bm\theta_0\|>\epsilon|\bm Y)=o_{P_0}(1).$
\end{lemma}
\begin{proof}
We denote the set $T_n=\{\|f-f_0\|\leq\tau_n\}$ for some $\tau_n\rightarrow0$. By Lemma 2, there exists such a sequence $\{\tau_n\}$ so that $\Pi(T^c_n|\bm Y)=o_{P_0}(1)$. By the triangle inequality, for $f\in T_n$,
\begin{eqnarray}
\lefteqn{|R_{ f}(\bm\eta)-R_{ f_0}(\bm\eta)|}\nonumber\\
&=&\left|\left\|f'(\cdot)-F(\cdot, f(\cdot), \bm\eta)\right\|_w-\left\|f_0'(\cdot)-F(\cdot, f_0(\cdot),\bm\eta)\right\|_w\right|\nonumber\\
&\leq&\left\|f'(\cdot)-f'_0(\cdot)\right\|_w+\left\|F(\cdot, f(\cdot), \bm\eta)-F(\cdot, f_0(\cdot), \bm\eta)\right\|_w\nonumber\\
&\leq&c_1\sup_{t\in[0,1]}\|f'(t)-f'_{0}(t)\|+c_2\sup_{t\in[0,1]}\|f(t)-f_{0}(t)\|,\nonumber
\end{eqnarray}
for appropriately chosen constants $c_1$ and $c_2$. Hence for $f\in T_n$,
\begin{eqnarray}
\sup_{\bm\eta\in\bm\Theta}|R_{f}(\bm\eta)-R_{f_0}(\bm\eta)|\leq c_1\sup_{t\in[0,1]}\|f'(t)-f'_{0}(t)\|+c_2\sup_{t\in[0,1]}\|f(t)-f_{0}(t)\|\nonumber
\end{eqnarray}
and for all $\epsilon>0$, $\Pi(\sup_{\bm\eta\in\bm\Theta}|R_{f}(\bm\eta)-R_{f_0}(\bm\eta)|>\epsilon\cap T_n|\bm Y)=o_{P_0}(1)$ due to Lemma 2. By definitions of $\bm\theta$ and $\bm\theta_0$, $R_{f}(\bm\theta)\leq R_{f}(\bm\theta_0)$. By assumption \eqref{assmp}, for $\|\bm\theta-\bm\theta_0\|\geq\epsilon$ there exists a $\delta>0$ such that
\begin{eqnarray}
\delta<R_{f_0}(\bm\theta)-R_{f_0}(\bm\theta_0)
&\leq&R_{f_0}(\bm\theta)-R_{f}(\bm\theta)+R_{f}(\bm\theta_0)-R_{f_0}(\bm\theta_0)\nonumber\\
&\leq&2\sup_{\bm\eta\in\bm\Theta}|R_{f}(\bm\eta)-R_{f_0}(\bm\eta)|.\nonumber
\end{eqnarray}
Consequently,
\begin{eqnarray}
\lefteqn{\Pi(\|\bm\theta-\bm\theta_0\|>\epsilon|\bm Y)}\nonumber\\
&\leq&\Pi(\sup_{\bm\eta\in\bm\Theta}|R_{f}(\bm\eta)-R_{f_0}(\bm\eta)|>{\delta}/{2}|\bm Y)\nonumber\\
&\leq&\Pi\left(\sup_{\bm\eta\in\bm\Theta}|R_{f}(\bm\eta)-R_{f_0}(\bm\eta)|>{\delta}/{2}\cap T_n|\bm Y\right)+\Pi(T^c_n|\bm Y)=o_{P_0}(1).\nonumber
\end{eqnarray}
\end{proof}
\begin{proof}[Proof of Theorem 1]
By definitions of $\bm\theta$ and $\bm\theta_0$,
\begin{eqnarray}
\int_0^1\left(D_{0,0,1}F(t, f(t),\bm\theta)\right)^T(f'(t)-F(t, f(t),\bm\theta))w(t)dt&=&0,\label{theta_1}\\
\int_0^1\left(D_{0,0,1}F(t, f_0(t),\bm\theta_0)\right)^T(f'_0(t)-F(t, f_0(t),\bm\theta_0))w(t)dt&=&0.\label{theta0}
\end{eqnarray}
The previous two equations lead to
\begin{eqnarray}
\lefteqn{\int_0^1((D_{0,0,1}F(t, f(t),\bm\theta))^T(f'_0(t)-F(t, f_0(t),\bm\theta_0))}\nonumber\\
&&-(D_{0,0,1}F(t, f(t),\bm\theta_0))^T(f'_0(t)-F(t, f_0(t),\bm\theta_0)))w(t)dt\nonumber\\
&&+\int_0^1(D_{0,0,1}F(t, f(t),\bm\theta_0)-D_{0,0,1}F(t, f_0(t),\bm\theta_0))^T(f'_0(t)-F(t, f_0(t),\bm\theta_0))w(t)dt\nonumber\\
&&+\int_0^1\left(D_{0,0,1}F(t, f(t),\bm\theta)\right)^T(f'(t)-f'_0(t)+F(t, f_0(t),\bm\theta_0)-F(t, f(t),\bm\theta_0)\nonumber\\
&&+F(t, f(t),\bm\theta_0)-F(t, f(t),\bm\theta))w(t)dt=0.\nonumber
\end{eqnarray}
Replacing the difference between the values of a function at two different values of an argument by the integral of the corresponding partial derivative, we finally get
\begin{eqnarray}
\lefteqn{\left(\int_0^1\left(D_{0,0,1}F(t, f(t),\bm\theta)\right)^T\left\{\int_0^1D_{0,0,1}F(t, f(t),\bm\theta+\lambda(\bm\theta_0-\bm\theta))d\lambda\right\} w(t)dt\right.}\nonumber\\
&&\left.-\int_0^1\left\{\int_0^1\left(D_{0,0,1}\bm S(t, f(t),\bm\theta_0+\lambda(\bm\theta-\bm\theta_0))\right)^Td\lambda\right\}w(t)dt\right)(\bm\theta-\bm\theta_0)\nonumber\\
&&=\int_0^1\left(D_{0,0,1}F(t, f(t),\bm\theta_0)-D_{0,0,1}F(t, f_0(t),\bm\theta_0)\right)^T(f'_0(t)-F(t, f_0(t),\bm\theta_0))w(t)dt\nonumber\\
&&+\int_0^1\left(D_{0,0,1}F(t, f(t),\bm\theta)\right)^T(f'(t)-f'_0(t)+F(t, f_0(t),\bm\theta_0)-F(t, f(t),\bm\theta_0))w(t)dt.\nonumber
\end{eqnarray}
Let us denote the matrix pre-multiplied to $\bm\theta-\bm\theta_0$ in the previous equation by $\bm M(f,\bm\theta)$. In particular, $\bm M(f_0,\bm\theta_0)=\bm J_{\bm\theta_0}$. We also define $E_n=\{\sup_{t\in[0,1]}\|f(t)-f_{0}(t)\|\leq\epsilon_n, \|\bm\theta-\bm\theta_0\|\leq\epsilon_n\}$, where $\epsilon_n\rightarrow0$. By Lemmas 2 and 3, there exists such a sequence $\{\epsilon_n\}$ so that $\Pi(E^c_n|\bm Y)=o_{P_0}(1)$. Then, $\bm M(f,\bm\theta)$ is invertible and the eigenvalues of $[\bm M(f,\bm\theta)]^{-1}$ are bounded away from $0$ and $\infty$ for sufficiently large $n$. Within $E_n$, we can represent $\sqrt{n}(\bm\theta-\bm\theta_0)$ in the posterior as
\begin{eqnarray}
(\bm M(f,\bm\theta))^{-1}\sqrt{n}(T_{1n}+T_{2n}+T_{3n}),\nonumber
\end{eqnarray}
where
\begin{eqnarray}
T_{1n}&=&\int_0^1\left(D_{0,0,1}F(t, f(t),\bm\theta_0)-D_{0,0,1}F(t, f_0(t),\bm\theta_0)\right)^T(f'_0(t)-F(t, f_0(t),\bm\theta_0))w(t)dt,\nonumber\\
T_{2n}&=&\int_0^1\left(D_{0,0,1}F(t, f(t),\bm\theta)\right)^T(f'(t)-f'_0(t))w(t)dt,\nonumber\\
T_{3n}&=&\int_0^1\left(D_{0,0,1}F(t, f(t),\bm\theta)\right)^T(F(t, f_0(t),\bm\theta_0)-F(t, f(t),\bm\theta_0))w(t)dt.\nonumber
\end{eqnarray}
It should be noted that $\|(\bm M(f,\bm\theta))^{-1}-\bm J^{-1}_{\bm\theta_0}\|=o(1)$ for $(\bm\theta, f)\in E_n$. We can assert that inside the set $E_n$, the asymptotic behavior of the posterior distribution of $\sqrt{n}(\bm\theta-\bm\theta_0)$ is given by that of
\begin{eqnarray}
{{\bm J^{-1}_{\bm\theta_0}}
\sqrt{n}(T_{1n}+T_{2n}+T_{3n})}.\label{asexpr}
\end{eqnarray}
We shall extract $\sqrt{n}\bm J^{-1}_{\bm\theta_0}\left(\bm\Gamma(f)-\Gamma(f_0)\right)$ from \eqref{asexpr} and show that the norm of the remainder is bounded up to a constant by the quantity on the right hand side of \eqref{thm1}. First let us deal with the second term inside the bracket above. We note that
\begin{eqnarray}
T_{2n}&=&\int_0^1\left(\bm D_{0,0,1}F(t, f_0(t),\bm\theta_0)\right)^T(f'(t)-f'_0(t))w(t)dt\nonumber\\
&&\,\,\,\,\,\,\,+\int_0^1\left(\bm D_{0,0,1}F(t, f(t),\bm\theta)-D_{0,0,1}F(t, f_0(t),\bm\theta_0)\right)^T(f'(t)-f'_0(t))w(t)dt\nonumber\\
&&=-\int_0^1\left(\frac{d}{dt}[(D_{0,0,1}F(t, f_0(t),\bm\theta_0))^Tw(t)]\right)(f(t)-f_0(t))dt\nonumber\\
&&\,\,\,\,\,\,\,+\int_0^1\left(\bm D_{0,0,1}F(t, f(t),\bm\theta)-D_{0,0,1}F(t, f_0(t),\bm\theta_0)\right)^T(f'(t)-f'_0(t))w(t)dt,\nonumber\\\label{exp2_1}
\end{eqnarray}
where the last equality follows by integration by parts and the fact that $w(0)=w(1)=0$. The first term on the right hand side of \eqref{exp2_1} appears in $\bm\Gamma(f)-\bm\Gamma(f_0)$. Now we consider the other term on the right-hand side of \eqref{exp2_1}.
By the Cauchy-Schwarz inequality,
\begin{eqnarray}
\lefteqn{\left\|\int_0^1(D_{0,0,1}F(t, f(t),\bm\theta)-D_{0,0,1}F(t, f_0(t),\bm\theta_0))^T(f'(t)-f'_0(t))w(t)dt\right\|}\nonumber\\
&&\leq\left\|D_{0,0,1}F(\cdot, f(\cdot), \bm\theta)-D_{0,0,1}F(\cdot, f_0(\cdot), \bm\theta_0)\right\|_w\left\|f'(\cdot)- f'_0(\cdot)\right\|_w.\nonumber\\\label{exp2_2}
\end{eqnarray}
Now by the continuity of $D_{0,1,1}F(t, y,\bm\theta)$ and $D_{0,0,2}F(t, y,\bm\theta)$, we have
\begin{eqnarray}
{\left\|D_{0,0,1}F(\cdot, f(\cdot), \bm\theta)-D_{0,0,1}F(\cdot, f_0(\cdot), \bm\theta_0)\right\|_w}
&\lesssim&\sup_{t\in[0,1]}\|f(t)-f_{0}(t)\|+\|\bm\theta-\bm\theta_0\|.\nonumber\\\label{exp2_3}
\end{eqnarray}
From \eqref{asexpr}, inside $E_n$, $\|\bm\theta-\bm\theta_0\|$ can be bounded above up to a constant by $\|T_{1n}\|+\|T_{2n}\|+\|T_{3n}\|$.
Applying the Cauchy-Schwarz inequality and utilizing the continuity of $D_{0,1,1}F(t, y,\bm\theta)$, we get
\begin{eqnarray}
\|T_{1n}\|+\|T_{2n}\|+\|T_{3n}\|&\lesssim&\sup_{t\in[0,1]}\|f'(t)-f'_{0}(t)\|\nonumber\\
&&+\sup_{t\in[0,1]}\|f(t)-f_{0}(t)\|.\label{exp2_4}
\end{eqnarray}
Combining \eqref{exp2_2}--\eqref{exp2_4}, we get
\begin{eqnarray}
%\begin{flalign*}
\lefteqn{\left\|\int_0^1(D_{0,0,1}F(t, f(t), \bm\theta)-D_{0,0,1}F(t, f_0(t), \bm\theta_0))^T(f'(t)-f'_0(t))w(t)dt\right\|}\nonumber\\
&\lesssim&\sup_{t\in[0,1]}\|f(t)-f_{0}(t)\|^2
+\sup_{t\in[0,1]}\|f'(t)-f'_{0}(t)\|^2.\label{exp2_8}
%\end{flalign*}
\end{eqnarray}
Now we consider $T_{3n}$ in \eqref{asexpr}. Then,
\begin{eqnarray}
T_{3n}&=&\int_0^1\left(D_{0,0,1}F(t, f_0(t),\bm\theta_0)\right)^T(F(t, f_0(t),\bm\theta_0)-F(t, f(t),\bm\theta_0))w(t)dt\nonumber\\
&&+\int_0^1\left(D_{0,0,1}F(t, f(t),\bm\theta)-D_{0,0,1}F(t, f_0(t),\bm\theta_0)\right)^T\nonumber\\
&&\times(F(t, f_0(t),\bm\theta_0)-F(t, f(t),\bm\theta_0))w(t)dt.\label{exp3_1}
\end{eqnarray}
The first term on the right hand side of \eqref{exp3_1} can be written as
\begin{eqnarray}
&-&\int_0^1\left(D_{0,0,1}F(t, f_0(t),\bm\theta_0)\right)^TD_{0,1,0}F(t, f_0(t),\bm\theta_0)(f(t)-f_0(t))w(t)dt\nonumber\\
&&-\int_0^1\left(D_{0,0,1}F(t, f_0(t),\bm\theta_0)\right)^T\nonumber\\
&&\times\left\{\int_0^1[D_{0,1,0}F(t, f_0(t)+\lambda(f-f_0)(t),\bm\theta_0)-D_{0,1,0}F(t, f_0(t),\bm\theta_0)]d\lambda\right\}\nonumber\\
&&\times(f(t)-f_0(t))w(t)dt\nonumber\\
&=&T_{31n}+T_{32n},\nonumber
\end{eqnarray}
say. Now $T_{31n}$ appears in $\bm\Gamma(f)-\bm\Gamma(f_0)$. By the continuity of $D_{0,2,0}F(t, y, \bm\theta)$, we obtain
\begin{eqnarray}
\|T_{32n}\|\lesssim\sup_{t\in[0,1]}\|f(t)-f_{0}(t)\|^2.\nonumber
%&&\left|\int_0^1[D_{0,1,0}F(t,f_0(t)+\lambda(f-f_0)(t),\bm\theta_0)-D_{0,1,0}F(t,f_0(t),\bm\theta_0)]d\lambda(f(t)-f_0(t))\right|\nonumber\\
%&&\leq\sup_{t\in[0,1]}|f(t)-f_{0}(t)|
%\int_0^1|D_{0,1,0}F(t,f_0(t)+\lambda(f-f_0)(t),\bm\theta_0)-D_{0,1,0}F(t,f_0(t),\bm\theta_0)|d\lambda\nonumber\\
%&&\leq\sup_{t\in[0,1]}|f(t)-f_{0}(t)|
%\int_0^1\sum_{i,j}|b_{i,j}(f_0(t)+\lambda(f(t)-f_0(t)),\bm\theta_0)-b_{i,j}(f_0(t),\bm\theta_0)|d\lambda\nonumber\\
%&&\leq\sup_{t\in[0,1]}|f(t)-f_{0}(t)|
%\sum_{i,j}\int_0^1\left|\left\{\int_0^1\frac{\partial}{\partial f}b_{i,j}(f_0(t)+\kappa\lambda(f(t)-f_0(t)),\bm\theta_0)d\kappa\right\}\lambda( %f(t)-f_0(t))\right|d\lambda\nonumber\\
%&&\leq\sup_{t\in[0,1]}|f(t)-f_{0}(t)|^2
%\sum_{i,j}\int_0^1\left\{\int_0^1\left|\frac{\partial}{\partial f}b_{i,j}(f_0(t)+\kappa\lambda(f(t)-f_0(t)),\bm\theta_0)\right|d\kappa\right\}d\lambda\nonumber,
\end{eqnarray}
Lastly, we treat the second term on the right hand side of \eqref{exp3_1}. By the Cauchy-Schwarz inequality, its norm can be bounded by
\begin{eqnarray}
{\left\|D_{0,0,1}F(\cdot, f(\cdot), \bm\theta)-D_{0,0,1}F(\cdot, f_0(\cdot), \bm\theta_0)\right\|_w}\left\|F(\cdot, f_0(\cdot), \bm\theta_0)-F(\cdot, f(\cdot), \bm\theta_0))\right\|_w\nonumber
\end{eqnarray}
The above expression can be shown to be bounded by a linear combination of $\sup\{\|f(t)-f_{0}(t)\|^2: t\in[0,1]\}$ and $\sup\{\|f'(t)-f'_{0}(t)\|^2: t\in[0,1]\}$ by the techniques used so far. As far as the first term inside the bracket of \eqref{asexpr} is concerned, we have
\begin{eqnarray}
T_{1n}&=&\int_0^1\left(D_{0,1,0}\bm S(t, f_0(t),\bm\theta_0)\right)(f(t)-f_0(t))w(t)dt\nonumber\\
&&+\int_0^1\left\{\int_0^1\left(D_{0,1,0}\bm S(t, f_0(t)+\lambda(f-f_0)(t),\bm\theta_0)-D_{0,1,0}\bm S(t, f_0(t),\bm\theta_0)\right)d\lambda\right\}\nonumber\\
&&\times(f(t)-f_0(t))w(t)dt\nonumber.
\end{eqnarray}
The first integral appears in $\bm\Gamma(f)-\bm\Gamma(f_0)$. The norm of the second integral of the above display can be bounded by a multiple of $\sup\{\|f(t)-f_{0}(t)\|^2: t\in[0,1]\}$ utilizing the continuity of $D_{0,2,1}F(t,y,\bm\theta)$ with respect to its arguments.
Combining these, we get \eqref{thm1}.
\end{proof}
Before proving Theorem 2, we need to study the asymptotic behaviors of the mean and variance of $\E(\bm G_{n,j}^T\bm\beta_j|\bm Y)$ for $j=1,\ldots,d$ given by the next lemma.
\begin{lemma}
Under the conditions of Theorem 2, the eigenvalues of $\bm\Var_0(\bm\E(\bm G_{n,j}^T\bm\beta_j|\bm Y))$ are of the order $n^{-1}$ and
\begin{eqnarray}
\max_{1\leq k\leq p}\left|[\bm\E_0(\bm\E(\bm G_{n,j}^T\bm\beta_j|\bm Y))]_k-\int_0^1 A_{k,j}(t)f_{j0}(t)dt\right|&=&o\left({n}^{-1/2}\right),\nonumber
\end{eqnarray}
where $A_{k,j}(t)$ denotes the $(k,j)^{th}$ element of the matrix $\bm A(t)$ as defined in Remark $2$ for $k=1,\ldots,p$, $j=1,\ldots,d$.
\end{lemma}
\begin{proof}
Let us fix $j\in\{1,\ldots,d\}$. We note that
\begin{eqnarray}
\bm\E(\bm G_{n,j}^T\bm\beta_j|\bm Y)={\left(1+\frac{k_n\sigma^2}{n}\right)}^{-1}\bm G_{n,j}^T{({\bm X^T_n\bm X_n})}^{-1}\bm X^T_n\bm Y_{,j}.\nonumber
\end{eqnarray}
Hence,
\begin{eqnarray}
\bm{\Var}_0(\bm\E(\bm G_{n,j}^T\bm\beta_j|\bm Y))=\sigma_0^2{\left(1+\frac{\sigma^2k_n}{n}\right)}^{-2}\bm G_{n,j}^T{({\bm X^T_n\bm X_n})}^{-1}\bm G_{n,j}\nonumber.
\end{eqnarray}
If $A_{k,j}(\cdot)\in C^{m^*}\left((0,1)\right)$ for some $1\leq m^*<m$, then by equation $(2)$ of \citet[page 167]{de1978practical}, we have $\sup\{|A_{k,j}(t)-\tilde A_{k,j}(t)|:t\in[0,1]\}=O(k_n^{-1})$, where $\tilde A_{k,j}(\cdot)=\bm \alpha_{k,j}^T\bm N(\cdot)$ and $\bm\alpha^T_{k,j}=(A_{k,j}(t^*_1),\ldots,A_{k,j}(t^*_{k_n+m-1}))$ with appropriately chosen $t^*_1,\ldots,t^*_{k_n+m-1}$.
We can write
\begin{eqnarray}
\bm G^T_{n,j}(\bm X^T_n\bm X_n)^{-1}\bm G_{n,j}&=&(\bm G_{n,j}-\tilde{\bm G}_{n,j})^T(\bm X^T_n\bm X_n)^{-1}(\bm G_{n,j}-\tilde{\bm G}_{n,j})\nonumber\\
&&+\tilde{\bm G}^T_{n,j}(\bm X^T_n\bm X_n)^{-1}\tilde{\bm G}_{n,j}+(\bm G_{n,j}-\tilde{\bm G}_{n,j})^T(\bm X^T_n\bm X_n)^{-1}\tilde{\bm G}_{n,j}\nonumber\\
&&+\tilde{\bm G}^T_{n,j}(\bm X^T_n\bm X_n)^{-1}(\bm G_{n,j}-\tilde{\bm G}_{n,j}),\nonumber
\end{eqnarray}
where $[\tilde{\bm G}_{n,j}]_{k,}=\int_0^1\tilde{A_{k,j}}(t)(\bm N(t))^Tdt$ for $k=1,\ldots,p.$ Let us denote $\tilde{\bm A}=(\!(\tilde A_{k,j})\!)_{1\leq k,j\leq p}$. We study the asymptotic orders of the eigenvalues of the matrices $\tilde{\bm G}^T_{n,j}(\bm X^T_n\bm X_n)^{-1}\tilde{\bm G}_{n,j}$ and $(\bm G_{n,j}-\tilde{\bm G}_{n,j})^T(\bm X^T_n\bm X_n)^{-1}(\bm G_{n,j}-\tilde{\bm G}_{n,j})$. It should be noted that
\begin{eqnarray}
\bm\alpha^T_{k,j}\int_0^1\bm N(t)\bm N^T(t)dt\,\bm\alpha_{k,j}&=&\int_0^1\tilde A^2_{k,j}(t)dt\nonumber\\
&\asymp&\|\bm\alpha_{k,j}\|^2k^{-1}_n,\nonumber
\end{eqnarray}
the last step following from the proof of Lemma 6.1 of \citet{zhou1998local}. Thus, the eigenvalues of $\int_0^1\bm N(t)(\bm N(t))^Tdt$ are of order $k^{-1}_n$. We introduce the notation $\lambda(\bm C)$ to denote an arbitrary eigenvalue of the matrix $\bm C$. Since the eigenvalues of $\left(\bm X^T_n\bm X_n/n\right)$ are of the order $k_n^{-1}$ \citep{zhou1998local}, we have
\begin{eqnarray}
\lambda\left(\tilde{\bm G}^T_{n,j}(\bm X^T_n\bm X_n)^{-1}\tilde{\bm G}_{n,j}\right)&\asymp&\frac{k_n}{n}\lambda\left(\tilde{\bm G}^T_{n,j}\tilde{\bm G}_{n,j}\right)\nonumber\\
&=&\frac{k_n}{n}\lambda\left(\int_0^1\tilde{\bm A_{,j}}(t)\bm N^T(t)dt\int_0^1\bm N(t)(\tilde{\bm A_{,j}}(t))^Tdt\right)\nonumber\\
&=&\frac{k_n}{n}\lambda\left(\left({\begin{array}{c}\bm\alpha_{1,j}^T\\\vdots\\\bm \alpha_{p,j}^T\end{array}}\right)\left(\int_0^1\bm N(t)\bm N^T(t)dt\right)^2\left(\bm\alpha_{1,j}\cdots\bm\alpha_{p,j}\right)\right)\nonumber\\
&\asymp&\frac{k_n}{nk_n^2}\lambda\left(\left({\begin{array}{ccc}\bm\alpha_{1,j}^T\bm\alpha_{1,j}&\cdots&\bm\alpha_{1,j}^T\bm\alpha_{p,j}\\\vdots&\ddots&\vdots\\
\bm\alpha_{1,j}^T\bm\alpha_{p,j}&\cdots&\bm\alpha_{p,j}^T\bm\alpha_{p,j}\end{array}}\right)\right)\nonumber\\
&\asymp&\frac{1}{n}\lambda\left(\left({\begin{array}{ccc}\left<A_{1,j}(\cdot),A_{1,j}(\cdot)\right>&\cdots&\left<A_{1,j}(\cdot),A_{p,j}(\cdot)\right>\\\vdots&\ddots&\vdots\\
\left<A_{1,j}(\cdot),A_{p,j}(\cdot)\right>&\cdots&\left<A_{p,j}(\cdot),A_{p,j}(\cdot)\right>\end{array}}\right)\right)\nonumber\\
&=&\frac{1}{n}\lambda\left(\bm B_j\right)\asymp\frac{1}{n}.\nonumber
\end{eqnarray}
Let us denote by $\bm 1_{k_n+m-1}$ the $k_n+m-1$-component vector with all elements $1$. Then for $k=1,\ldots,p$,
\begin{eqnarray}
\lefteqn{\left[(\bm G_{n,j}-\tilde{\bm G}_{n,j})^T(\bm X^T_n\bm X_n)^{-1}(\bm G_{n,j}-\tilde{\bm G}_{n,j})\right]_{k,k}}\nonumber\\
&\asymp&\int_0^1(A_{k,j}(t)-\tilde A_{k,j}(t))(\bm N(t))^Tdt\,{({\bm X^T_n\bm X_n})}^{-1}\int_0^1(A_{k,j}(t)-\tilde A_{k,j}(t))(\bm N(t))dt\nonumber\\
&=&\frac{1}{n}\int_0^1(A_{k,j}(t)-\tilde A_{k,j}(t))(\bm N(t))^Tdt\,{\left({\bm X^T_n\bm X_n}/{n}\right)}^{-1}\int_0^1(A_{k,j}(t)-\tilde A_{kj}(t))\bm N(t)dt\nonumber\\
&\asymp&\frac{k_n}{n}\int_0^1(A_{k,j}(t)-\tilde A_{k,j}(t))(\bm N(t))^Tdt\int_0^1(A_{k,j}(t)-\tilde A_{k,j}(t))\bm N(t)dt\nonumber\\
&\lesssim&\frac{k_n}{nk_n^2}\int_0^1\bm 1^T_{k_n+m-1}\bm N(t)\bm N^T(t)dt\int_0^1\bm N(t)\bm N^T(t)\bm 1_{k_n+m-1}dt\nonumber\\
&=&\frac{1}{nk_n}\bm 1^T_{k_n+m-1}\left(\int_0^1\bm N(t)\bm N^T(t)dt\right)^2\bm 1_{k_n+m-1}\nonumber\\
&\asymp&\frac{1}{nk_n}\bm 1^T_{k_n+m-1}\bm 1_{k_n+m-1}k_n^{-2}\asymp\frac{1}{nk_n^2}.\nonumber
\end{eqnarray}
Thus, the eigenvalues of $(\bm G_{n,j}-\tilde{\bm G}_{n,j})^T(\bm X^T_n\bm X_n)^{-1}(\bm G_{n,j}-\tilde{\bm G}_{n,j})$ are of the order $(nk_n^2)^{-1}$ or less. Hence, the eigenvalues of $\bm G^T_{n,j}(\bm X^T_n\bm X_n)^{-1}\bm G_{n,j}$ are of the order $n^{-1}$.\\
Similar to the proof of Lemma 1, we can write for the $\bm\beta^*_j$ given in \eqref{spldis},
\begin{eqnarray}
\lefteqn{\sqrt{n}\left|[\bm\E_0(\bm\E(\bm G_{n,j}^T\bm\beta_j|\bm{Y}))]_k-\int_0^1 A_{k,j}(t)f_{j0}(t)dt\right|}\nonumber\\
&\leq&\sqrt{n}\left|{\left(1+\frac{k_n\sigma^2}{n}\right)}^{-1}[\bm G_{n,j}^T\bm\beta^*_j]_k-[\bm G_{n,j}^T\bm\beta^*_j]_k\right|\nonumber\\
&&+\sqrt{n}{\left(1+\frac{k_n\sigma^2}{n}\right)}^{-1}\left|[\bm G_{n,j}^T{(\bm X^T_n\bm X_n)}^{-1}\bm X^T_n(f_{j0}(\bm{x})-\bm X_n\bm\beta^*_j)]_k\right|\nonumber\\
&&+\sqrt{n}\left|\int_0^1 A_{k,j}(t)f_{j0}(t)dt-[\bm G_{n,j}^T\bm\beta^*_j]_k\right|,\nonumber
\end{eqnarray}
where $[\bm G_{n,j}^T\bm\beta^*_j]_k=\int_0^1A_{k,j}(t)f^*_j(t)dt$ and $f^*_j(t)=\bm N^T(t)\bm\beta^*_j$ for $k=1,\ldots,p$. Proceeding in the same way as in the proof of Lemma 1, we can show that each term on the right hand side of the above equation converges to zero.\end{proof}
Since the bias and standard deviation of $[\bm\E(\bm G^T_{n,j}\bm\beta|\bm Y)]_k$ are of order $n^{-1/2}$ as an estimator of $\int_0^1A_{k,j}(t)f_{j0}(t)dt$, we can assert that for $j=1,\ldots,d$, and $k=1,\ldots,p$,
\begin{eqnarray}
\left|[\bm\E(\bm G^T_{n,j}\bm\beta|\bm Y)]_k-\int_0^1A_{k,j}(t)f_{j0}(t)dt\right|=O_{P_0}(n^{-1/2}).\label{tight}
\end{eqnarray}
\begin{proof}[Proof of Theorem 2]
We observe that for $j=1,\ldots,d$,
\begin{eqnarray}
\frac{\|f_{j0}(\bm x)\|^2}{n^2k_n^{-2}}=\frac{k_n^2}{n}\frac{\|f_{j0}(\bm x)\|^2}{n}
=O\left(\frac{k_n^2}{n}\right)=o(1),\nonumber
\end{eqnarray}
and ${(k_n+m-1)k_n^2}/{n^2}\asymp k_n^3/n^2=o(1)$. Now proceeding in the same way as in the proofs of Theorem 1 and Corollary 1 of \citet{bontemps2011bernstein}, we get for $j\neq j';\,j,j'=1,\ldots,d$,
\begin{eqnarray}
\left\|\Pi(\bm G_{n,j}^T\bm\beta_j\in\cdot|\bm Y)-{N}(\bm G_{n,j}^T(\bm X^T_n\bm X_n)^{-1}\bm X^T_n\bm Y_{,j},\sigma^2\bm G_{n,j}^T(\bm X^T_n\bm X_n)^{-1}\bm G_{n,j})\right\|_{TV}=o_{P_0}(1).\nonumber
\end{eqnarray}
Since the posterior distributions of $\bm\beta_j$ and $\bm\beta_j'$ are mutually independent for $j\neq j';\,j,j'=1,\ldots,d$,
\begin{eqnarray}
\left\|\Pi\left(\sum_{j=1}^d\bm G_{n,j}^T\bm\beta_j\in\cdot|\bm Y\right)-{N}\left(\sum_{j=1}^d\bm G_{n,j}^T(\bm X^T_n\bm X_n)^{-1}\bm X^T_n\bm Y_{,j},\sigma^2\sum_{j=1}^d\bm G_{n,j}^T(\bm X^T_n\bm X_n)^{-1}\bm G_{n,j}\right)\right\|_{TV}\nonumber\\
=o_{P_0}(1),\nonumber
\end{eqnarray}
and hence
\begin{eqnarray}
\left\|\Pi\left(\sqrt{n}\sum_{j=1}^d\bm G_{n,j}^T\bm\beta_j-\sqrt{n}\bm J^{-1}_{\bm\theta_0}\bm\Gamma(\bm f_0)\in\cdot|\bm Y\right)-{N}(\bm\mu_n,\sigma^2\bm\Sigma_n)\right\|_{TV}
=o_{P_0}(1).\nonumber\\\label{thm2_1}
\end{eqnarray}
By \eqref{thm1} and \eqref{linearize}, we may write for $(\bm\theta, \bm f)\in E_n$,
\begin{eqnarray}
\sqrt{n}(\bm\theta-\bm\theta_0)=\sqrt{n}\left(\sum_{j=1}^d\bm G_{n,j}^T\bm\beta_j-\bm J^{-1}_{\bm\theta_0}\bm\Gamma(\bm f_0)\right)+\sqrt{n}\,{\bm \eta}_n,\nonumber
\end{eqnarray}
where $\|\bm\eta_n\|\lesssim\sup\{\|\bm f(t)-\bm f_0(t)\|^2: t\in[0,1]\}+\sup\{\|\bm f'(t)-\bm f'_0(t)\|^2: t\in[0,1]\}$. By Lemma 2, for all $\epsilon>0$, $\Pi(\sqrt{n}\|{\bm\eta}_n\|>\epsilon|\bm Y)=o_{P_0}(1)$. We fix $E^*_n=\{\sqrt{n}\|\bm\eta_n\|\leq\delta_n\}$. There exists a $\delta_n\rightarrow0$ such that $\Pi(E_n\cap E^*_n|\bm Y)\stackrel{P_0}\rightarrow 1$. Now for any $A\in\mathscr{R}^p$,
\begin{eqnarray}
\Pi(\{\sqrt{n}(\bm\theta-\bm\theta_0)\in A\}\cap E_n\cap E^*_n|\bm Y)&\leq&\Pi(\sqrt{n}(\bm\theta-\bm\theta_0)\in A|\bm Y)\nonumber\\
&\leq&\Pi(\{\sqrt{n}(\bm\theta-\bm\theta_0)\in A\}\cap E_n\cap E^*_n|\bm Y)\nonumber\\
&&+\Pi((E_n\cap E^*_n)^c|\bm Y).\label{thm2_2}
\end{eqnarray}
Again using \eqref{linearize} and the triangle inequality,
\begin{eqnarray}
\lefteqn{\sup_{A\in\mathscr{R}^p}\left|\Pi(\{\sqrt{n}(\bm\theta-\bm\theta_0)\in A\}\cap E_n\cap E^*_n|\bm Y)-P(N_p(\bm\mu_n,\sigma^2\bm\Sigma_n)\in A)\right|}\nonumber\\
&\leq&\sup_{\|\bm u\|\leq\delta_n}\sup_{A\in\mathscr{R}^p}\left|\Pi\left(\sqrt{n}\left(\sum_{j=1}^d\bm G_{n,j}^T\bm\beta_j-\bm J^{-1}_{\bm\theta_0}\bm\Gamma(\bm f_0)\right)\in A-\bm u|\bm Y\right)\right.\nonumber\\
&&\left.-P(N_p(\bm\mu_n,\sigma^2\bm\Sigma_n)\in A-\bm u)\right|\nonumber\\
&&+\sup_{\|\bm u\|\leq\delta_n}\sup_{A\in\mathscr{R}^p}\left|P(N_p(\bm\mu_n,\sigma^2\bm\Sigma_n)\in A-\bm u)-P(N_p(\bm\mu_n,\sigma^2\bm\Sigma_n)\in A)\right|\nonumber\\
&\leq&\sup_{B\in\mathscr{R}^p}\left|\Pi\left(\sqrt{n}\left(\sum_{j=1}^d\bm G_{n,j}^T\bm\beta_j-\bm J^{-1}_{\bm\theta_0}\bm\Gamma(\bm f_0)\right)\in B|\bm Y\right)-P(N_p(\bm\mu_n,\sigma^2\bm\Sigma_n)\in B)\right|\nonumber\\
&&+\frac{1}{\sqrt{2\pi}}\delta_n.\nonumber\\\label{thm2_3}
\end{eqnarray}
In the last step we used the fact that the eigenvalues of $\bm\Sigma_n$ are bounded away from zero and infinity as obtained from Lemma 4. The first term of the right hand side of \eqref{thm2_3} converges in probability to zero by \eqref{thm2_1}. Combining this with \eqref{thm2_2}, we get the desired result.
\end{proof}
\begin{proof}[Proof of Theorem 3]
According to the fitted model, $\bm Y_{i,}^{1\times d}\sim N_d((\bm X_n)_{i,}\bm B_n,\Sigma^{d\times d})$ for $i=1,\ldots,n$. The logarithm of the posterior probability density function (p.d.f.) is proportional to
\begin{eqnarray}
\sum_{i=1}^n\left((\bm X_n)_{i,}(\bm\beta_1\ldots\bm\beta_d)-\bm Y_{i,}\right)\bm\Sigma^{-1}\left(\left({\begin{array}{c}\bm \beta_1^T\\\vdots\\\bm \beta_d^T\end{array}}\right)(\bm X^T_n)_{,i}-\bm Y^T_{i,}\right)\nonumber\\
+\sum_{j=1}^d\bm\beta_j^T\frac{\bm X^T_n\bm X_n}{nk_n^{-1}}\bm\beta_j.\label{logpost}
\end{eqnarray}
First let us consider the part of the above expression quadratic in $\bm\beta_j$, $j=1,\ldots,d$. We have,
\begin{eqnarray}
\lefteqn{\sum_{i=1}^n(\bm X_n)_{i,}(\bm\beta_1\ldots\bm\beta_d)\bm\Sigma^{-1}\left({\begin{array}{c}\bm \beta_1^T\\\vdots\\\bm \beta_d^T\end{array}}\right)(\bm X^T_n)_{,i}}\nonumber\\
&=&\mathrm{tr}\left(\left({\begin{array}{c}(\bm X_n)_{1,}\\\vdots\\\bm (\bm X_n)_{n,}\end{array}}\right)(\bm\beta_1\ldots\bm\beta_d)\bm\Sigma^{-1}\left({\begin{array}{c}\bm \beta_1^T\\\vdots\\\bm \beta_d^T\end{array}}\right)\left((\bm X^T_n)_{,1}\ldots(\bm X^T_n)_{,n}\right)\right)\nonumber\\
&=&\mathrm{tr}\left(\bm\Sigma^{-1}\left({\begin{array}{c}\bm \beta_1^T\\\vdots\\\bm \beta_d^T\end{array}}\right)\bm X^T_n\bm X_n(\bm\beta_1\ldots\bm\beta_d)\right).\nonumber\\\label{thm3_1}
\end{eqnarray}
Again,
\begin{equation}
\sum_{j=1}^d\bm\beta_j^T\frac{\bm X^T_n\bm X_n}{nk_n^{-1}}\bm\beta_j=\mathrm{tr}\left(\left({\begin{array}{c}\bm \beta_1^T\\\vdots\\\bm \beta_d^T\end{array}}\right)\frac{\bm X^T_n\bm X_n}{nk_n^{-1}}(\bm\beta_1\ldots\bm\beta_d)\right).\label{thm3_2}
\end{equation}
Adding \eqref{thm3_1} and \eqref{thm3_2}, the part of \eqref{logpost} quadratic in $\bm\beta_j$, $j=1,\ldots,d$ becomes
\begin{equation}
\mathrm{tr}\left(\left(\bm\Sigma^{-1}+\frac{k_n\bm I_d}{n}\right)\left({\begin{array}{c}\bm \beta_1^T\\\vdots\\\bm \beta_d^T\end{array}}\right)\bm X^T_n\bm X_n(\bm\beta_1\ldots\bm\beta_d)\right).\label{thm3_3}
\end{equation}
Now we deal with the part of \eqref{logpost} linear in $\bm\beta_j$, $j=1,\ldots,d$. Then,
\begin{eqnarray}
\sum_{i=1}^n(\bm X_n)_{i,}(\bm\beta_1\ldots\bm\beta_d)\bm\Sigma^{-1}\bm Y^T_{i,}
&=&\mathrm{tr}\left(\left({\begin{array}{c}(\bm X_n)_{1,}\\\vdots\\\bm (\bm X_n)_{n,}\end{array}}\right)(\bm\beta_1\ldots\bm\beta_d)\bm\Sigma^{-1}(\bm Y^T_{1,}\ldots\bm Y^T_{n,})\right)\nonumber\\
&=&\mathrm{tr}\left(\bm\Sigma^{-1}(\bm Y^T_{1,}\ldots\bm Y^T_{n,})\left({\begin{array}{c}(\bm X_n)_{1,}\\\vdots\\\bm (\bm X_n)_{n,}\end{array}}\right)(\bm\beta_1\ldots\bm\beta_d)\right).\label{thm3_4}
\end{eqnarray}
We can rewrite \eqref{thm3_4} as
\begin{equation}
\mathrm{tr}\left(\left(\bm\Sigma^{-1}+\frac{k_n\bm I_d}{n}\right)\left(\bm\Sigma^{-1}+\frac{k_n\bm I_d}{n}\right)^{-1}\bm\Sigma^{-1}(\bm Y^T_{1,}\ldots\bm Y^T_{n,})\bm X_n(\bm X^T_n\bm X_n)^{-1}\bm X^T_n\bm X_n\bm B_n\right).\nonumber
\end{equation}
Then the logarithm of the posterior p.d.f. is proportional to
\begin{eqnarray}
&&\mathrm{tr}\left(\left(\bm\Sigma^{-1}+\frac{k_n\bm I_d}{n}\right)\bm B^T_n\bm X^T_n\bm X_n\bm B_n\right)\nonumber\\
&&-2\,\mathrm{tr}\left(\left(\bm\Sigma^{-1}+\frac{k_n\bm I_d}{n}\right)\left(\bm\Sigma^{-1}+\frac{k_n\bm I_d}{n}\right)^{-1}\bm\Sigma^{-1}(\bm Y^T_{1,}\ldots\bm Y^T_{n,})\bm X_n(\bm X^T_n\bm X_n)^{-1}\bm X^T_n\bm X_n\bm B_n\right).\nonumber
\end{eqnarray}
Normalizing, the posterior density turns out to be proportional to
\begin{eqnarray}
\mathrm{exp}\left\{-\frac{1}{2}\mathrm{tr}\left[\left(\bm\Sigma^{-1}+\frac{k_n\bm I_d}{n}\right)\left(\bm B_n-(\bm X^T_n\bm X_n)^{-1}\bm X^T_n(\bm Y^T_{1,}\ldots\bm Y^T_{n,})^T\bm\Sigma^{-1}\left(\bm\Sigma^{-1}+\frac{k_n\bm I_d}{n}\right)^{-1}\right)^T\right.\right.\nonumber\\
\left.\left.\bm X^T_n\bm X_n\left(\bm B_n-(\bm X^T_n\bm X_n)^{-1}\bm X^T_n(\bm Y^T_{1,}\ldots\bm Y^T_{n,})^T\bm\Sigma^{-1}\left(\bm\Sigma^{-1}+\frac{k_n\bm I_d}{n}\right)^{-1}\right)\right]\right\},\nonumber
\end{eqnarray}
and hence the posterior distribution is a matrix normal distribution. Then, the posterior distribution of $\mathrm{vec}(\bm B_n)$ is a $(k_n+m-1)d$-dimensional normal distribution with mean vector $\mathrm{vec}\left((\bm X^T_n\bm X_n)^{-1}\bm X^T_n(\bm Y^T_{1,}\ldots\bm Y^T_{n,})^T\bm\Sigma^{-1}\left(\bm\Sigma^{-1}+\frac{k_n\bm I_d}{n}\right)^{-1}\right)$ and dispersion matrix $\left(\bm\Sigma^{-1}+\frac{k_n\bm I_d}{n}\right)^{-1}\otimes(\bm X^T_n\bm X_n)^{-1}$. Fixing a $j\in\{1,\ldots,d\}$, we observe that the posterior mean of $\bm\beta_j$ is a weighted sum of $(\bm X^T_n\bm X_n)^{-1}\bm X^T_n\bm Y_{,j'}$ for $j'=1,\ldots,d$. The weight attached with $(\bm X^T_n\bm X_n)^{-1}\bm X^T_n\bm Y_{,j}$ is of the order of $1$, whereas for $j'\neq j$, the contribution from $(\bm X^T_n\bm X_n)^{-1}\bm X^T_n\bm Y_{,j'}$ is of the order of $k_n/n$ which goes to zero as n goes to infinity. Thus, the results of Lemmas 1, 2 and 3 can be shown to hold under this setup.
We are interested in the limiting distribution of
\begin{eqnarray}
\bm J^{-1}_{\bm\theta_0}\bm\Gamma(\bm f)=\sum_{j=1}^d\bm G_{n,j}^T\bm\beta_j=(\bm G_{n,1}^T\ldots\bm G_{n,d}^T)\mathrm{vec}(\bm B_n).\nonumber
\end{eqnarray}
We note that the posterior distribution of $\left(\left(\bm\Sigma^{-1}+\frac{k_n\bm I_d}{n}\right)^{1/2}\otimes\bm I_{k_n+m-1}\right)\mathrm{vec}(\bm B_n)$ is a $(k_n+m-1)d$-dimensional normal distribution with mean vector and dispersion matrix being $\mathrm{vec}\left((\bm X^T_n\bm X_n)^{-1}\bm X^T_n(\bm Y^T_{1,}\ldots\bm Y^T_{n,})^T\bm\Sigma^{-1}\left(\bm\Sigma^{-1}+\frac{k_n\bm I_d}{n}\right)^{-1/2}\right)$ and $\bm I_d\otimes(\bm X^T_n\bm X_n)^{-1}$ respectively, since by the properties of Kronecker product, for the matrices $\bm A$, $\bm B$ and $\bm D$ of appropriate orders
\begin{eqnarray}
(\bm B^T\otimes\bm A)\mathrm{vec}(\bm D)=\mathrm{vec}(\bm{ADB}).\nonumber
\end{eqnarray}
Let us consider the mean vector of the posterior distribution of the vector $\left(\left(\bm\Sigma^{-1}+\frac{k_n\bm I_d}{n}\right)^{1/2}\otimes\bm I_{k_n+m-1}\right)\mathrm{vec}(\bm B_n)$. We observe that
\begin{eqnarray}
\lefteqn{(\bm X^T_n\bm X_n)^{-1}\bm X^T_n(\bm Y^T_{1,}\ldots\bm Y^T_{n,})^T\bm\Sigma^{-1}\left(\bm\Sigma^{-1}+\frac{k_n\bm I_d}{n}\right)^{-1/2}}\nonumber\\
&=&(\bm X^T_n\bm X_n)^{-1}\bm X^T_n(\bm Y_{,1}\ldots\bm Y_{,d})\bm\Sigma^{-1}\left(\bm\Sigma^{-1}+\frac{k_n\bm I_d}{n}\right)^{-1/2}\nonumber\\
&=&(\bm X^T_n\bm X_n)^{-1}\bm X^T_n(\bm Y_{,1}\ldots\bm Y_{,d})\left(\bm \Sigma+\frac{k_n\bm\Sigma^2}{n}\right)^{-1/2}\nonumber\\
&=&(\bm X^T_n\bm X_n)^{-1}\bm X^T_n\left(\sum_{j=1}^d\bm Y_{,j}c_{j1}\ldots\sum_{j=1}^d\bm Y_{,j}c_{jd}\right)^{n\times d},\nonumber
\end{eqnarray}
where $\bm C^{d\times d}=(\!(c_{jk})\!)=\left(\bm \Sigma+\frac{k_n\bm\Sigma^2}{n}\right)^{-1/2}$. If a random vector $\bm W$ follows a multivariate normal distribution, then any sub-vector follows a multivariate normal distribution with mean vector and dispersion matrix being the corresponding sub-vector of $\bm{\E}(\bm W)$ and the corresponding sub-matrix of $\bm{\Var}(\bm W)$ respectively. Thus, for $k=1,\ldots,d$,
\begin{eqnarray}
\bm Z^{(k_n+m-1)\times 1}_k&:=&\left[\left(\left(\bm\Sigma^{-1}+\frac{k_n\bm I_d}{n}\right)^{1/2}\otimes\bm I_{k_n+m-1}\right)\mathrm{vec}(\bm B_n)\right]_{(k-1)(k_n+m-1)+1:k(k_n+m-1)}\nonumber\\\label{thm3_5}
\end{eqnarray}
and $\bm Z_k|\bm Y\sim N_{k_n+m-1}\left({(\bm X^T_n\bm X_n)}^{-1}\bm X^T_n\sum_{j=1}^d\bm Y_{,j}c_{jk},{(\bm X^T_n\bm X_n)}^{-1}\right)$. Also, the posterior distributions of $\bm Z_k$ and $\bm Z_{k'}$ are mutually independent for $k\neq k';k,k'=1,\ldots,d$.
Now we will prove that the total variation distance between the posterior distribution of $\bm Z_k$ and $N\left({(\bm X^T_n\bm X_n)}^{-1}\bm X^T_n\sum_{j=1}^d\bm Y_{,j}\sigma^{jk},{(\bm X^T_n\bm X_n)}^{-1}\right)$ converges in probability to zero for $k=1,\ldots,d$, where $\bm\Sigma^{-1/2}=(\!(\sigma^{jk})\!)$. We find that
\begin{eqnarray}
&\E_0&\left\|N\left({(\bm X^T_n\bm X_n)}^{-1}\bm X^T_n\sum_{j=1}^d\bm Y_{,j}c_{jk},{(\bm X^T_n\bm X_n)}^{-1}\right)\right.\nonumber\\
&&\left.-N\left({(\bm X^T_n\bm X_n)}^{-1}\bm X^T_n\sum_{j=1}^d\bm Y_{,j}\sigma^{jk},{(\bm X^T_n\bm X_n)}^{-1}\right)\right\|^2_{TV}\nonumber\\
&\lesssim&\E_0\left(\sum_{j=1}^d\|{(\bm X^T_n\bm X_n)}^{-1/2}\bm X^T_n\bm Y_{,j}(c_{jk}-\sigma^{jk})\|\right)^2\nonumber\\
&\lesssim&\sum_{j=1}^d\E_0\|{(\bm X^T_n\bm X_n)}^{-1/2}\bm X^T_n\bm Y_{,j}(c_{jk}-\sigma^{jk})\|^2.\label{thm3_6}
\end{eqnarray}
Fixing $k$, for $j=1,\ldots,d$, we have that
\begin{eqnarray}
\|{(\bm X^T_n\bm X_n)}^{-1/2}\bm X^T_n\bm Y_{,j}(c_{jk}-\sigma^{jk})\|^2&=&(c_{jk}-\sigma^{jk})^2\bm Y_{,j}^T\bm X_n{(\bm X^T_n\bm X_n)}^{-1}\bm X^T_n\bm Y_{,j}\nonumber\\
&\leq&(c_{jk}-\sigma^{jk})^2\bm Y_{,j}^T\bm Y_{,j},\nonumber
\end{eqnarray}
since the eigenvalues of $\bm X_n{(\bm X^T_n\bm X_n)}^{-1}\bm X^T_n$ are either zero or $1$. If $\lambda$ is an eigenvalue of $\bm\Sigma$, the corresponding eigenvalues of $\bm\Sigma^{-1/2}$ and $\bm C$ are $\lambda^{-1/2}$ and $\left(\lambda+\frac{k_n\lambda^2}{n}\right)^{-1/2}$ respectively. Let $\bm u$ be the eigenvector of $\bm\Sigma$ corresponding to $\lambda$, that is, $\bm\Sigma\bm u=\lambda\bm u$. Then,
\begin{eqnarray}
\left(\bm\Sigma^{-1/2}-\bm C\right)\bm u=\bm\Sigma^{-1/2}\bm u-\bm C\bm u=\lambda^{-1/2}\bm u-\left(\lambda+\frac{k_n\lambda^2}{n}\right)^{-1/2}\bm u.\nonumber
\end{eqnarray}
Thus, $\bm\Sigma^{-1/2}-\bm C$ is positive definite with the corresponding eigenvalue $\lambda^{-1/2}-\left(\lambda+\frac{k_n\lambda^2}{n}\right)^{-1/2}$ which is of the order $k_n/n$. Now using the Cauchy-Schwarz inequality, we get
\begin{eqnarray}
|c_{jk}-\sigma^{jk}|&=&|(\bm\Sigma^{-1/2}-\bm C)_{j,k}|\nonumber\\
&=&|\bm e^T_j(\bm\Sigma^{-1/2}-\bm C)\bm e_k|\nonumber\\
&\leq&|(\bm\Sigma^{-1/2}-\bm C)_{j,j}(\bm\Sigma^{-1/2}-\bm C)_{k,k}|^{1/2}\nonumber\\
&\leq&\lambda_{\max}(\bm\Sigma^{-1/2}-\bm C)\asymp\frac{k_n}{n},\nonumber
\end{eqnarray}
where $\bm e_j$ is a $d$-dimensional vector with $1$ at the $j^{th}$ position and zero elsewhere and $\lambda_{\max}(\bm A)$ stands for the maximum eigenvalue of the matrix $\bm A$. Hence for $j=1,\ldots,d$,
\begin{eqnarray}
\E_0\|{(\bm X^T_n\bm X_n)}^{-1}\bm X^T_n\bm Y_{,j}(c_{jk}-\sigma^{jk})\|^2&\leq&(c_{jk}-\sigma^{jk})^2\E_0(\bm Y_{,j}^T\bm Y_{,j})\nonumber\\
&\lesssim&\frac{k_n^2}{n^2}\times n=\frac{k_n^2}{n}=o(1).\label{thm3_8}
\end{eqnarray}
Combining \eqref{thm3_6} and \eqref{thm3_8}, we conclude that
\begin{eqnarray}
&&\left\|N\left({(\bm X^T_n\bm X_n)}^{-1}\bm X^T_n\sum_{j=1}^d\bm Y_{,j}c_{jk},{(\bm X^T_n\bm X_n)}^{-1}\right)
\right.\nonumber\\
&&\left.-N\left({(\bm X^T_n\bm X_n)}^{-1}\bm X^T_n\sum_{j=1}^d\bm Y_{,j}\sigma^{jk},{(\bm X^T_n\bm X_n)}^{-1}\right)\right\|_{TV}=o_{P_0}(1).\nonumber\\\label{thm3_9}
\end{eqnarray}
It should be noted that
\begin{eqnarray}
\lefteqn{\left(\bm G_{n,1}^T\ldots\bm G_{n,d}^T\right)\mathrm{vec}(\bm B_n)}\nonumber\\
&=&\left(\bm G_{n,1}^T\ldots\bm G_{n,d}^T\right)\left(\left(\bm\Sigma^{-1}+\frac{k_n\bm I_d}{n}\right)^{1/2}\otimes\bm I_{k_n+m-1}\right)^{-1}\nonumber\\
&&\times\left(\left(\bm\Sigma^{-1}+\frac{k_n\bm I_d}{n}\right)^{1/2}\otimes\bm I_{k_n+m-1}\right)\mathrm{vec}(\bm B_n)\nonumber\\
&=&\left(\bm G_{n,1}^T\ldots\bm G_{n,d}^T\right)\left(\left(\bm\Sigma^{-1}+\frac{k_n\bm I_d}{n}\right)^{1/2}\otimes\bm I_{k_n+m-1}\right)^{-1}\left({\begin{array}{c}\bm Z_1\\\vdots\\\bm Z_d\end{array}}\right)\nonumber\\
&=&\sum_{k=1}^d\left[\left(\bm G_{n,1}^T\ldots\bm G_{n,d}^T\right)\left(\left(\bm\Sigma^{-1}+\frac{k_n\bm I_d}{n}\right)^{1/2}\otimes\bm I_{k_n+m-1}\right)^{-1}\right]_{,{(k-1)(k_n+m-1)+1:k(k_n+m-1)}}\nonumber\\
&&\times\bm Z_k.\label{thm3_10}
\end{eqnarray}
Let us denote
\begin{eqnarray}
\bm\mu^{**}_n&=&\sqrt{n}\sum_{k=1}^d\left[\left(\bm G^T_{n,1}\ldots\bm G^T_{n,d}\right)\left(\left(\bm\Sigma^{-1}+\frac{k_n\bm I_d}{n}\right)^{1/2}\otimes\bm I_{k_n+m-1}\right)^{-1}\right]_{,{(k-1)(k_n+m-1)+1:k(k_n+m-1)}}\nonumber\\
&&\times{(\bm X^T_n\bm X_n)}^{-1}\bm X_n^T\sum_{j=1}^d\bm Y_j\sigma^{jk}
-\bm J_{\bm\theta_0}^{-1}\sqrt{n}\bm\Gamma(\bm f_0),\nonumber\\
\bm\Sigma^{**}_n&=&n\sum_{k=1}^d\left[\left(\bm G^T_{n,1}\ldots\bm G^T_{n,d}\right)\left(\left(\bm\Sigma^{-1}+\frac{k_n\bm I_d}{n}\right)^{1/2}\otimes\bm I_{k_n+m-1}\right)^{-1}\right]_{,{(k-1)(k_n+m-1)+1:k(k_n+m-1)}}\nonumber\\
&&\times{(\bm X^T_n\bm X_n)}^{-1}\nonumber\\
&&\times\left[\left(\bm G^T_{n,1}\ldots\bm G^T_{n,d}\right)\left(\left(\bm\Sigma^{-1}+\frac{k_n\bm I_d}{n}\right)^{1/2}\otimes\bm I_{k_n+m-1}\right)^{-1}\right]^T_{{(k-1)(k_n+m-1)+1:k(k_n+m-1)},}.\nonumber
\end{eqnarray}
Since the posterior distributions of $\bm Z_k$, $k=1,\ldots,d$ are independent, \eqref{thm3_9} and \eqref{thm3_10} give
\begin{eqnarray}
\,\,\,\,\,\,\,\,\,\,\,\,\,\,\,\,\,\,\,\,\,\,\,\,\,\,\,\left\|\left(\sqrt{n}\left(\bm G_{n,1}^T\ldots\bm G_{n,d}^T\right)\mathrm{vec}(\bm B_n)-\sqrt{n}\bm J^{-1}_{\bm\theta_0}(\bm f_0)\right)-N(\bm\mu_n^{**},\bm\Sigma_n^{**})\right\|_{TV}=o_{P_0}(1).\nonumber
\end{eqnarray}
Following the steps of the proof of Lemma 4, it can be shown that the eigenvalues of the matrix $\bm\Sigma^*_n$ mentioned in the statement of Theorem 3 are of the order $n^{-1}$. We can show that the Kullback-Leibler divergence of $N(\bm\mu^{**}_n,\bm\Sigma^{**}_n)$ from $N(\bm\mu^{*}_n,\bm\Sigma^{*}_n)$ converges in probability to zero by going through some routine matrix manipulations. Hence,
\begin{eqnarray}
\,\,\,\,\,\,\,\,\,\,\,\,\,\,\,\,\,\,\,\,\,\,\,\,\,\,\,\left\|\left(\sqrt{n}\left(\bm G_{n,1}^T\ldots\bm G_{n,d}^T\right)\mathrm{vec}(\bm B_n)-\sqrt{n}\bm J^{-1}_{\bm\theta_0}(\bm f_0)\right)-N(\bm\mu_n^{*},\bm\Sigma_n^{*})\right\|_{TV}=o_{P_0}(1).\nonumber
\end{eqnarray}
The above expression is equivalent to \eqref{thm2_1} of the proof of Theorem 2. The rest of the proof follows in the same way as the last part of the proof of Theorem 2.
\end{proof}

\bibliographystyle{chicago}
\bibliography{ref}

\end{document}